\documentclass[10pt,a4paper]{article}
\usepackage{amsfonts}
\usepackage[latin9]{inputenc}
\usepackage{amsmath}
\usepackage{amssymb}
\usepackage{breqn}
\usepackage{url}
\usepackage{graphicx}
\usepackage[pdfstartview=FitH]{hyperref}
\usepackage{amsthm}
\usepackage{hyperref}
\usepackage{url}
\usepackage{enumitem}
\usepackage{authblk}

\newcommand{\im}{\operatorname{Im}}
\newcommand{\Ker}{\operatorname{Ker}}

\newcommand{\Stab}{\operatorname{Stab}}

\makeatletter

\begin{document}
\newtheorem{theorem}{Theorem}[section]
\newtheorem{lemma}[theorem]{Lemma}
\newtheorem{definition}[theorem]{Definition}
\newtheorem{claim}[theorem]{Claim}
\newtheorem{example}[theorem]{Example}
\newtheorem{remark}[theorem]{Remark}
\newtheorem{proposition}[theorem]{Proposition}
\newtheorem{corollary}[theorem]{Corollary}
\newtheorem{observation}[theorem]{Observation}
\newcommand{\subscript}[2]{$#1 _ #2$}
\newtheorem*{theorem*}{Theorem}

\author{Izhar Oppenheim}
\affil{Department of Mathematics, Ben-Gurion University of the Negev, Be'er Sheva 84105, Israel, izharo@bgu.ac.il}
\title{Vanishing of cohomology with coefficients in representations on Banach spaces of groups acting on Buildings}
\maketitle
\begin{abstract}
We prove vanishing of cohomology with coefficients in representations on a large class of Banach spaces for a group acting ``nicely'' on a simplicial complexes based on spectral properties of the 1-dimensional links of the simplicial complex.
\end{abstract}

\section{Introduction}

The study of group actions on metric space is a broad topic in which one studies the interplay between the group structure and the structure of the metric space on which it acts. When considering a group action on Hilbert spaces, property (FH) is an imporant notion which is defined as follows: a group $G$ has property (FH) if every isometric action of $G$ on a Hilbert space admits a fixed point. Property (FH) can be rephrased in a cohomological language as follows: a group has property (FH) if and only if $H^1 (G,\pi)=0$ for any unitary representation of $G$ on a Hilbert space (the proof of this fact can be found for instance in \cite{PropTBook}[Lemma 2.2.6]).

Recently there have been much interest in studying the generalization of property (FH) for group actions on Banach spaces (see \cite{Nowak2} for a survey of recent developments regarding this question). In order to state a generalization of property (FH) in the Banach setting, we recall the following facts taken from \cite{Nowak2}[Section 2.2]: 
\begin{itemize}
\item Any affine action $A$ on a Banach space $X$ is of the form $Ax = Tx+b$ where $T$ is a linear map and $b \in X$.
\item As a result of the previous fact, if $\rho$ defines an affine action of $G$ on $X$, then 
$$\forall g \in G, \rho (g).x = \pi (g).x + b(g),$$
where $\pi : G \rightarrow B(X)$ is a linear representation of $G$ on $X$ called the linear part of $\rho$ and $b:G \rightarrow X$ is a map which satisfies the cocycle condition: 
$$\forall g,h \in G, b(gh) = \pi (g) b(h) + b(g).$$
\item For a group $G$ and a linear representation $\pi$ on a Banach space $X$, $H^1 (G,\pi)=0$ if and only if any affine action with a linear part $\pi$ admits a fixed point.
\end{itemize}

Thus the vanishing of the first cohomology reflects a rigidity phenomenon. In this article, we will explore a generalization of this phenomenon and we will prove vanishing of higher cohomologies for groups acting on simplicial complexes. The idea is that given a ``nice enough'' group action on a simplicial complex $\Sigma$, one can show vanishing of cohomology with coefficients in representations on Banach spaces under suitable assumption on the norm growth of the action and on the geometry of the simplicial complex. This is done by using the interplay between the geometry of the Banach space and the geometry of the simplicial complex as it is reflected in the angles between couples of subgroups of $G$ stabilizing top-dimensional simplices in the simplicial complex (see definition below).

The definition of angle between subgroups in the Hilbert setting is as follows: let $G$ be a group, $\pi$ by a unitary representation of $G$ on a Hilbert space $H$ and let $K_1, K_2 < G$ be subgroups of $G$. The angle between $K_1$ and $K_2$ with respect to $\pi$ is defined as the (Friedrichs) angle between $H^{\pi (K_1)}$ and $H^{\pi (K_2)}$. The angle between $K_1$ and $K_2$ is then defined as the supremum with respect to all unitary representations of $G$. 

This idea of angle was used in the work of Dymara and Januszkiewicz \cite{DymaraJ} to prove property (T) (which is equivalent to property (FH) in this setting) and vanishing of higher cohomologies with coefficients in representations on Hilbert spaces for groups acting on simplicial complexes. Dymara and Januszkiewicz further showed how to bound the angle between the two subgroups using the spectral gap of the Laplacian on a graph generated by these subgroups.

At first glance, this idea seems very much related to the so called ``geometrization of property (T)'' (this term was coined by Shalom \cite{Shalom}), since it uses the spectral gap of a Laplacian to deduce property (T) in a way similar to Zuk's famous criterion for property (T) (see \cite{Zuk}, \cite{BallmannS}). However, at its core, the idea of angle between subgroups is much stronger than Zuk's criterion, because it better captures the behaviour of the group $G$. In \cite{ORobust} the author generalized this idea of angle to the setting of Banach spaces, considering angle between projections instead of angle between subspaces. This new notion of angle was used by the author it to show a strengthened version of Banach property (T) for a large class of Banach spaces. This in turn implies the vanishing of the first group cohomology with coefficients in the isometric representations on this class of Banach spaces. 

The aim of this paper is to generalize the vanishing of cohomologies theorem of Dymara and Januszkiewicz in \cite{DymaraJ} to coefficients in representations on Banach spaces. A major problem with transfering the results of Dymara and Januszkiewicz to Banach space setting was that the angle computations of \cite{DymaraJ}[section 4] heavily relied on the idea that in Hilbert spaces the angle between two subspaces is equal to the angle between their orthogonal complements. However, this idea of computing the angle by passing to the orthogonal complement does not seem to work in our definition of angle between projections. 

The technical heart of this paper is devoted to attaining results regarding angles between projections  in Banach spaces that are similar to the results of Dymara and Januszkiewicz (but without passing to the orthogonal complemet). In order to attain these results, we first explore the idea of angle between more than two projections (this was inspired by the ideas of Kassabov in \cite{Kassabov}). 

After obtaining these technical results, the vanishing theorem can be reproved for coefficients in representations on Banach spaces by the same arguments given in \cite{DymaraJ}. 

In order to apply these results in concrete examples (such as groups groups coming from a BN-pair), we need to bound angles between pairs of subgroups $K_1,K_2 <G$ with respect to representations on Banach spaces. Given a pair of subgroups $K_1,K_2 <G$, this is done by bounding this angle between these subgroups in the Hilbert setting and then (if this angle is large enough to begin with) using this bound in order to get a bound on the angles between these subgroups with respect to representations on Banach spaces that are ``close enough'' to a Hilbert space. Being ``close enough'' to a Hilbert space involves a several step process of deforming a Hilbert space that will be explained in detail below. 

\subsection{Deformations of Hilbert spaces}

We will consider Banach spaces which are deformations of Hilbert spaces. In order to explain which deformations we consider, we need to introduce several ideas from the theory of Banach spaces.

\subsubsection{The Banach-Mazur distance and Banach spaces with ``round'' subspaces}
The Banach-Mazur distance measures a distance between isomorphic Banach spaces:

\begin{definition}
Let $Y_1, Y_2$ be two isomorphic Banach spaces. The (multiplicative) Banach-Mazur distance between $Y_1$ and $Y_2$ is defined as 
$$d_{BM} (Y_1, Y_2) = \inf \lbrace \Vert T \Vert \Vert T^{-1} \Vert : T : Y_1 \rightarrow Y_2 \text{ is a linear isomorphism} \rbrace.$$
\end{definition}

This distance has a multiplicative triangle inequality (the proof is left as an exercise to the reader):
\begin{proposition}
Let $Y_1,Y_2,Y_3$ be isomorphic Banach spaces. Then
$$d_{BM}(Y_1,Y_3) \leq d_{BM}(Y_1,Y_2) d_{BM}(Y_2,Y_3).$$
\end{proposition} 

We will be especially interested in the Banach-Mazur distance between $n$-dimensional Banach spaces and $\ell_2^n$. A classical theorem by F. John \cite{John} states for every $n$-dimensional Banach space $Y$, $d_{BM} (Y, \ell_2^n) \leq \sqrt{n}$ and the classical cases in which this inequality is an equality are $\ell_1^n$ and $\ell_\infty^n$. Later, Milman and Wolfson \cite{MilmanWolfson} proved that these classical cases are in some sense generic: \cite{MilmanWolfson}[Theorem 1] states that if $d_{BM} (Y, \ell_2^n) = \sqrt{n}$, then there is $k\geq \frac{\ln (n)}{2 \ln (12)}$ such that $Y$ contains a $k$-dimensional subspace isometric to $\ell_1^k$. 

In this paper we will concern ourselves with Banach spaces whose finite dimensional subspaces are sufficiently ``round'', i.e., sufficiently close to $\ell_2$-spaces. Given a Banach space $X$ and a constant $k \in \mathbb{N}$, we use the following notation $d_k (X)$ taken from the work of de Laat and de la Salle \cite{LaatSalle2}:
$$d_k (X) = \sup \lbrace d_{BM} (Y) : Y \subseteq X, \dim (Y) \leq k \rbrace.$$

We further introduce the following notation: given a constant $r \geq 2$ and a constant $C_1 \geq 1$, we denote $\mathcal{E}_1 (r,C_1)$ to be the class of Banach spaces defined as follows:
\begin{align*}
\mathcal{E}_1 (r, C_1 ) = 
 \lbrace X : \forall k \in \mathbb{N}, d_k (X) \leq C_1 k^{\frac{1}{r}} \rbrace.
\end{align*}

The reader should note that for every choice of $r \geq 2, C_1 \geq 1$, the class $\mathcal{E}_1 (r, C_1 )$ always contains the class of all Hilbert spaces, since for every Hilbert space $H$ we have that $d_k (H) =1$ for every $k$. 

An example of Banach spaces contained in $\mathcal{E}_1 (r,C_1)$ are spaces of bounded type and cotype. The definitions of  type and cotype are given in the background section below, but for our uses, it is sufficient to state the following theorem due to Tomczak-Jaegermann \cite{Tomczak-Jaegermann}[Theorem 2 and the corollary after it]: if $X$ is a Banach space of type $p_1$, cotype $p_2$ and corresponding constants $T_{p_1} (X)$, $C_{p_2} (X)$ (see definitions below), then $d_k (X) \leq 4 T_{p_1} (X) C_{p_2} (X) k^{\frac{1}{p_1} - \frac{1}{p_2}}$. 

This theorem yields that for every $r>2$, every $\frac{1}{r}$  and every $C_1 \geq 1$, the class $\mathcal{E}_1 (r,C_1 )$ contains all Banach spaces $X$ with type $p_1$, cotype $p_2$ and corresponding constants $T_{p_1} (X)$, $C_{p_2} (X)$  such that $\frac{1}{p_1} - \frac{1}{p_2} \leq \frac{1}{r}$ and $4 T_{p_1} (X) C_{p_2} (X) \leq C_1$.     

\subsubsection{Interpolation}

Two Banach spaces $X_0, X_1$ form a compatible pair $(X_0,X_1)$ if there is a continuous linear embedding of both $X_0$ and $X_1$ in the same topological vector space. The idea of complex interpolation is that given a compatible pair $(X_0,X_1)$ and a constant $0 < \theta < 1$, there is a method to produce a new Banach space $[X_0, X_1]_\theta$ as a ``combination'' of $X_0$ and $X_1$. We will not review this method here, and the interested reader can find more information on interpolation in \cite{InterpolationSpaces}.

We will introduce the following notation: let $\mathcal{E}$ be a class of Banach spaces and let $0 < \theta_2 \leq 1$ be a constant. Denote $\mathcal{E}_2 (\mathcal{E}, \theta_2)$ the class of Banach spaces defined as follows
\begin{align*}
\mathcal{E}_2 (\mathcal{E},\theta_2)= \lbrace X : \exists X_1 \in \mathcal{E}, \exists X_0 \text{ Banach}, \exists \theta_2 \leq \theta \leq 1 \text{ such that } X=[X_0,X_1]_\theta \rbrace.
\end{align*} 

We will be interested with composing this definition with $\mathcal{E}_1 (r,C_1)$ defined above and considering $\mathcal{E}_2 (\mathcal{E}_1 (r,C_1),\theta_2)$. 

As noted above $\mathcal{E}_1 (r,C_1)$ contains the class of all the Hilbert spaces. This brings us to consider the following definition is due to Pisier in \cite{Pisier}: a Banach space $X$ is called strictly $\theta$-Hilbertian for $0 < \theta \leq 1$, if there is a compatible pair $(X_0,X_1)$ such that $X_1$ is a Hilbert space such that $X = [X_0, X_1]_\theta$. Examples of strictly $\theta$-Hilbertian spaces are $L^p$ space and non-commutative $L^p$ spaces, where in these cases $\theta = \frac{2}{p}$ if $2 \leq  p  < \infty $ and $\theta = 2 - \frac{2}{p}$ if $1 < p \leq 2$ (a reader who is not familiar with non-commutative $L^p$ spaces can find a detailed account in \cite{PisierXu2}). 

Another source of examples for strictly $\theta$-Hilbertian spaces are superreflexive Banach lattices. Recall that a Banach space $X$ is called uniformly convex if 
$$\sup \left\lbrace \left\Vert \frac{x+y}{2} \right\Vert : \forall x,y \in X, \Vert x \Vert = \Vert y \Vert =1, \Vert x - y \Vert \geq \varepsilon \right\rbrace <1 \text{ for every } \varepsilon >0.$$
Further recall that a Banach space $X$ is called superreflexive if all its ultrapowers are reflexive, which is equivalent by \cite{BenLin}[Theorem A.6] to $X$ being isomorphic to a uniformly convex space. A Banach lattice is a Banach space with a ``well-behaved'' partial order on it - the definition is rather techincal and we will not recall it here (for the exact defition of a Banach lattice and further properties of it, the reader is reffered to \cite{BenLin}[Appendix G]). 

Pisier \cite{Pisier} proved that any superreflexive Banach lattice is strictly $\theta$-Hilbertian and suggested that this result might by true even for superreflexive Banach spaces which are not Banach lattices.

\subsubsection{Passing to a isomorphic space}
The last deformation we want to consider is passing to an isomorphic space. We introduce the following notation: let $\mathcal{E}$ be a class of Banach spaces and let $C_3 \geq 1$ be a constant, denote by $\mathcal{E}_3 (\mathcal{E},C_3)$ the class of Banach spaces defined as 
\begin{align*}
\mathcal{E}_3 (\mathcal{E}, C_3) = \lbrace X : \exists X' \in \mathcal{E} \text{ such that } d_{BM} (X,X') \leq C_3 \rbrace.
\end{align*}

\subsubsection{Passing to the closure}
Our criterion for vanishing of cohomology relies on geometric properties of a Banach space that are stable under certain operations. Therefore, we can enlarge our Banach class by passing to the closure under these operations: for a class of Banach spaces $\mathcal{E}$, denote by $\overline{\mathcal{E}}$ the smallest class of Banach spaces containing $\mathcal{E}$ that is closed under passing to quotients, subspaces, $l_2$-sums and ultraproducts.

\subsubsection{Composing the deformations}

The class of Banach spaces we will want to consider is the composition of the all the deformations described above, i.e., we start with a Hilbert space and use $\mathcal{E}_1 (r, C_1)$ to consider deformations of it, on that class we consider interpolation, then pass to isomorphic spaces with bounded Banach-Mazur distance and finish by passing to the closure. 
To put it all together, we start with constants $r \geq 2$, $C_1 \geq 1$, $1 \geq \theta_2 >0$ and $C_3 \geq 1$ and consider the class $\overline{\mathcal{E}_3 (\mathcal{E}_2 (\mathcal{E}_1 (r,C_1),\theta_2),C_3)}$.

\subsection{The main theorem for BN-pair groups}

Following Dymara and Januszkiewicz, our vanishing of cohomology results are true for groups acting on simplicial complexes given that certain conditions are fulfilled (conditions $(\mathcal{B}1)-(\mathcal{B} 4)$ and $(\mathcal{B}_{\delta, r})$ stated below). However, currently, our only examples of groups acting on complexes satisfying these conditions are groups with a BN-pair (e.g., classical BN-pair groups acting on Euclidean buildings or 2-spherical Kac-Moody groups). Therefore, in this introduction we will state our main result only for BN-pair groups (the more general Theorem \ref{vanishing of cohomology by conditions on the links and cont dual assumption} is given below). 

In order to state the main theorem, we recall some generalities regarding BN-pair groups (a reader not familiar with BN-pair groups can find and extensive treatment of this subject in \cite{BuildingsBook}[Chapter 6]) and introduce a few notations regarding representations. 

Let $G$ be a BN-pair group and let $\Sigma$ be the $n$-dimensional building on which it acts. Then $G$ acts on $\Sigma$ cocompactly and $\triangle = \Sigma / G$ is a single chamber of $\Sigma$. We assume that $n >1$, i.e., that $\Sigma$ is not a tree and denote $\triangle (k)$ to be the $k$-dimensional faces of $\Sigma / G$. We assume further that there is some $l \in \mathbb{N}$ such that all the $l$-dimensional links of $\Sigma$ are compact.  Be this assumption, for every $\tau \in \triangle (n-2)$, the isotropy group $G_\tau = \Stab (\tau)$ is compact and $G$ is generated by $\bigcup_{\tau \in \triangle (n-2)} G_\tau$. Let $\mathcal{E}$ be a class of Banach spaces and let $s_0 >0$ be a constant. Denote $\mathcal{F} (\mathcal{E}, G, s_0)$ to be all the continuous representations $(\pi,X)$ of $G$ such that $X \in \mathcal{E}$ and 
$$\sup_{g \in \bigcup_{\tau \in \triangle (n-2)} G_\tau} \Vert \pi (g) \Vert \leq e^{s_0}.$$ 
Note that $\mathcal{F} (\mathcal{E}, G, s_0)$ contains all the isometric representation of $G$ on some $X \in \mathcal{E}$, but is also contains representations which are not uniformly bounded. Indeed, if $G$ is taken with the word norm $\vert . \vert$ with respect to $\bigcup_{\tau \in \triangle (n-2)} G_\tau$, then $\mathcal{F} (\mathcal{E}, G, s_0)$ contains all the representations $\pi$ such that $\Vert \pi (g) \Vert \leq e^{s_0 \vert g \vert}$ for every $g \in G$. Denote further $\mathcal{F}_0 (\mathcal{E}, G, s_0)$ to be 
$$\mathcal{F}_0 (\mathcal{E}, G, s_0) = \lbrace \pi \in \mathcal{F} (\mathcal{E}, G, s_0) : \pi^* \text{ is a continuous representation} \rbrace,$$
where $\pi^*$ is the dual representation of $\pi$.

After all these notations and definitions, we are ready to state our main theorem:

\begin{theorem*}
Let $G$ be a group coming from a BN-pair and let $\Sigma$ be the $n$-dimensional building on which it acts. Assume that $n >1$ and there is some $l \in \mathbb{N}$ such that all the $l$-dimensional links of $\Sigma$ are compact. Denote by $q+1$ the thickness of the building $\Sigma$.

Let $r>20$, $C_1 \geq 1$, $1 \geq \theta_2 >0$, $C_3 \geq 1$ be constants. Then there are constants $s_0 = s_0 (n)$ and $Q = Q(n, C_1 , \theta_2, C_3)$ such that if $q \geq Q$, then for every 
$$\pi \in \mathcal{F}_0 (\overline{\mathcal{E}_3 (\mathcal{E}_2 (\mathcal{E}_1 (r, C_1),\theta_2),C_3)}, G, s_0),$$
we have that
$$H^i (G,\pi) = 0, i=1,...,l.$$
\end{theorem*}  

\begin{remark}
In the above theorem, when considering $\mathcal{E}_1 (r, C_1)$ we took $r>20$. In many cases, this choice can be improved, i.e., $r$ can be taken to be smaller, if the codimension $1$ links of the building $\Sigma$ are known. For instance, if $\Sigma$ is known to be an $\widetilde{A}_n$ building, then we have the same theorem above with $r>4$. The precise statement of this fact is given in Corollary \ref{vanishing of cohomology for BN pair groups}
\end{remark}

\subsection{Examples of Banach spaces for which the theorem holds}

In the main theorem above, we considered only representations whose dual is continuous. This might seem to be a major restriction, but we will show below that the class of representations that we are considering is still very rich. We will do so by showing that there are interesting examples (families of) of Banach spaces in $\overline{\mathcal{E}_3 (\mathcal{E}_2 (\mathcal{E}_1 (r, C_1),\theta_2),C_3)}$ with $r>20$ for which each continuous representation has a continuous dual. 

Indeed, \cite{Megre}[Corollary 6.9] states that if $X$ is an Asplund Banach space then for every continuous representation $\pi$, the dual representation $\pi^*$ is also continuous. The exact definition of Asplund spaces in given in the next section (along with a good reference regarding these spaces), but for our needs, it is enough to recall that any reflexive Banach space is an Asplund space. Using this fact, we will show that $\overline{\mathcal{E}_3 (\mathcal{E}_2 (\mathcal{E}_1 (r, C_1),\theta_2),C_3)}$ contains many interesting reflexive spaces.       

First, for a Banach space $X$, we recall that $X$ is called uniformly non-square if there is some $\varepsilon >0$ such that for every $x, y \in X$ with in the unit ball of $X$, $\min \lbrace \Vert \frac{x+y}{2} \Vert, \Vert \frac{x-y}{2} \Vert \rbrace \leq 1- \varepsilon$. James \cite{James}[Theorem 1.1] showed that every uniformly non-square space is reflexive. An easy exercise shows that if $d_2 (X) < \sqrt{2}$ then $X$ is uniformly non-square. Therefore, for every $X \in \mathcal{E}_1 (r, C_1)$, if $d_2 (X) < \sqrt{2}$, then $X$ is reflexive, i.e., every $X \in \mathcal{E}_1 (r, C_1)$ whose $2$-dimensional subspaces are not too distorted is a reflexive space. 

Second, since $\mathcal{E}_1 (r, C_1)$ contains all Hilbert space, we have that $\mathcal{E}_2 (\mathcal{E}_1 (r, C_1),\theta_2)$ contains all $\theta$-Hilbertian spaces with $\theta \geq \theta_2$. As noted above this includes $L^p$ spaces and non commutative $L^p$ spaces with $\frac{2}{2-\theta_2} \leq p \leq \frac{2}{\theta_2}$. By \cite{PisierXu2}[Theorem 5.1] these spaces are uniformly convex and therefore superreflexive (hence reflexive). Also, $\mathcal{E}_2 (\mathcal{E}_1 (r, C_1),\theta_2)$ also includes a subclass of the class of superreflexive Banach lattices (as noted above, for any superreflexive Banach lattice $X$, there is $\theta >0$ such that $X$ is $\theta$-Hilbertian).

Third, reflexivity of Banach spaces is preserved under isomorphism and therefore $\mathcal{E}_3 (\mathcal{E}_2 (\mathcal{E}_1 (r, C_1),\theta_2),C_2)$ contains isomorphic spaces to the reflexive Banach spaces contained in $\mathcal{E}_2 (\mathcal{E}_1 (r, C_1),\theta_2)$.

Last, reflexivity is preserved under passing to a closed subspace, taking a quotient by a closed subspace and countable $l_2$-sums. Ultrapowers does not preserve reflexivity, but by definition, if $X$ is superreflexive, then all its ultrapowers are reflexive. Therefore passing to the closure $\overline{\mathcal{E}_3 (\mathcal{E}_2 (\mathcal{E}_1 (r, C_1),\theta_2),C_3)}$ provides more examples of reflexive Banach spaces constructed from $\mathcal{E}_3 (\mathcal{E}_2 (\mathcal{E}_1 (r, C_1),\theta_2),C_3)$ by these operations. 

\begin{remark}
Above we have a list of families of Banach spaces (e.g., $L^p$ spaces or uniformly non-square spaces in $\mathcal{E} (r,C_1)$ when $r>20$) for which our main theorem holds for every representation in which the norm doesn't grow too fast with respect to the word norm (in particular, for every isometric representation). As far as we know, for each one of these families our vanishing of higher cohomologies results are new even in the classical case of BN-pair groups acting on Euclidean buildings. 
\end{remark}

\textbf{Structure of this paper.} Section 2 includes all the needed background material. Section 3 is devoted to proving the main technical result regarding angles between projections in Banach spaces. In section 4, we formulate and prove our main results regarding vanishing of cohomologies for groups acting on simplicial complexes. The appendix contains technical results regarding angles between projections under a weaker assumptions that the ones used in section 3, that may be of independent interest.

\section{Background}
\subsection{Groups acting on simplicial complexes}
\label{Groups acting on simplicial complexes subscetion}
Here we present the set up needed for our results of groups acting on simplicial complexes.  We start by recalling some definitions given by Dymara and Januszkiewicz in \cite{DymaraJ}[section 1]. 
 
Let $\Sigma$ be a countable pure $n$-dimensional simplicial complex with $n \geq 2$. The top dimensional simplices of $\Sigma$ will be called chambers and $\Sigma$ will be called gallery connected if for any two chambers $\sigma, \sigma'$ there is a sequence of chambers 
$$\sigma = \sigma_1, \sigma_2,...,\sigma_k = \sigma',$$
such that for every $i$, $\sigma_i \cap \sigma_{i+1}$ is a simplex of co-dimension $1$ in $\Sigma$. 
 
Denote by $Aut (\Sigma)$ the group of simplicial automorphisms of $\Sigma$. On $Aut (\Sigma)$ define the compact-open topology whose basis are the sets $U(K,g_0)$ where $g_0 \in Aut (\Sigma)$, $K \subseteq \Sigma$ compact and $U(K,g_0)$ is defined as
$$U(K,g_0) =  \lbrace g \in Aut (\Sigma) : g \vert_{K} = g_0 \vert_K \rbrace.$$
Let $G < Aut (\Sigma)$ be a closed subgroup of $Aut (\Sigma)$. 

Given a continuous representation $\pi$ of $G$ on a Banach space, one can define $H^* (G, \pi)$ and $H^* (\Sigma, \pi)$. We will not review these definitions here and a reader unfamilier with these definitions can find them in \cite{DymaraJ}[Section 3] and reference therein. The main fact that we will use is that one can compute $H^* (G, \pi)$ based on $H^* (\Sigma, \pi)$:
\begin{lemma}\cite{BorelW}[X.1.12]
\label{locally finite simplicial complex cohomology lemma}
Let $\Sigma$ be a simplicial complex, $G < Aut (\Sigma)$ be a closed subgroup and $\pi$ be a representation of $G$ on a Banach space. Assume that $\Sigma$ is contractible and locally finite and that the action of $G$ on $\Sigma$ is cocompact, then $H^* (G, \pi) = H^* (\Sigma, \pi)$.
\end{lemma}

The above lemma assumes that $\Sigma$ is locally finite (i.e., that the link of every vertex is compact). In order to compute the cohomology of $G$ in cases where $\Sigma$ is not locally finite, Dymara and Januszkiewicz introduced the following definition of the core of $\Sigma$:

\begin{definition}\cite{DymaraJ}[Definition 1.3]
Let $\Sigma$ be a simplicial complex such that every link of $\Sigma$ is either compact or contractible (including $\Sigma$ itself, which is the link of the empty set) and such that the $0$-dimensional links of $\Sigma$ are finite. Denote $\Sigma '$ to be the first barycentric subdivision of $\Sigma$. The core  of $\Sigma$, denoted $\Sigma_D$, is the subcomplex of $\Sigma'$ spanned by the barycenters of simplices of $\Sigma$ with compact links. 
\end{definition}

\begin{lemma}\cite{DymaraJ}[Lemma 1.4]
Let $\Sigma$ be an infinite simplicial complex such that every link of $\Sigma$ is either compact or contractible (in particular $\Sigma$ is contractible, because it is the link of $\emptyset$) and such that the $0$-dimensional links of $\Sigma$ are finite. Then $\Sigma_D$ is contractible.
\end{lemma}

Note that if the assumption that the $0$-dimensional links of $\Sigma$ are finite implies that $\Sigma_D$ is locally finite. Also note that any closed subgroup $G < Aut (\Sigma)$ is also a closed subgroup in $Aut (\Sigma_D)$. Therefore combining the above lemma with Lemma \ref{locally finite simplicial complex cohomology lemma} above yields the following corollary:
\begin{corollary}
Let $\Sigma$ be an infinite pure $n$-dimensional simplicial complex, $G < Aut (\Sigma)$ be a closed subgroup and $\pi$ be a representation of $G$ on a Banach space. Assume that every link of $\Sigma$ is either compact or contractible and such that the $0$-dimensional links of $\Sigma$ are finite. If the action of $G$ on $\Sigma$ is cocompact, then $H^* (G, \pi) = H^* (\Sigma_D, \pi)$.
\end{corollary}

Following Dymara and Januszkiewicz, we will use the above corollary to show vanishing of the group cohomology under additional assumptions on $\Sigma$ and on the action of $G$. In order to state our additional assumptions we recall the following conditions on the couple $(\Sigma, G)$ taken from \cite{DymaraJ}: 
\begin{enumerate}[label=($\mathcal{B}${{\arabic*}})]
\item All the $0$-dimensional links are finite.
\item All the links of dimension $\geq 1$ are gallery connected. 
\item All the links are either compact or contractable (including $\Sigma$ itself). 
\item $G$ acts transitively on chambers and $\Sigma \rightarrow \Sigma / G$ restricts to an isomorphism on every chamber. 
\end{enumerate}

Let $\Sigma$ be an infinite simplicial complex and $G < Aut (\Sigma)$ be a closed subgroup satisfying $(\mathcal{B}1)-(\mathcal{B} 4)$ and let $\pi$ a continuous representation of $G$ on a Banach space $X$. Fix a chamber $\triangle \in \Sigma (n)$ and for every $\eta \subseteq \triangle$, denote $G_\eta$ to be the subgroup of $G$ fixing $\sigma$ and also denote $X^{\pi (G_\eta)} = X_\eta$ to be the subspace of $X$ fixed by $G_\eta$ (under the action of $\pi$).  One of the key ideas in \cite{DymaraJ} is that one can deduce vanishing of cohomologies of $G$ with coefficients in $\pi$ given that there are projections on all the $X_\eta$'s and nice decompositions of these $X_\eta$'s with respect to these projections. To make this precise:

\begin{theorem}\cite{DymaraJ}[Theorems 5.2,7.1]
\label{general vanishing of cohomology based on decomposition}
Let $\Sigma$ be an infinite simplicial complex, $G < Aut (\Sigma)$ be a closed subgroup satisfying $(\mathcal{B}1)-(\mathcal{B} 4)$ and $\pi$ a continuous representation of $G$ on a Banach space $X$. Under the notations above, for every $\eta \subseteq \triangle$ denote $D_\eta$ to be the subcomplex of $\Sigma_D$ spanned by the barycenters of simplices of $\triangle$ that have compact links and do not contain $\eta$. 
 
Assume that for every $\eta \subseteq \triangle$ there is a projection $P_\eta :X \rightarrow X$ on $X_\eta$. For every $\eta \subseteq \triangle$, denote
$$X^\eta = \im (P_\eta) \cap \bigcap_{\tau \subsetneqq \eta} \Ker (P_\tau).$$
If for every $\eta \subseteq \triangle$, the following holds 
$$X_\eta = \bigoplus_{\tau \subseteq \eta} X^\tau,$$
then 
$$H^* (G, \pi) = \bigoplus_{\eta \subseteq \triangle} \widetilde{H}^{*-1} (D_\eta; X^{\eta}).$$ 
Moreover, if there is $l \geq 1$ such that all the $l$-dimensional links of $\Sigma$ are compact, then for every $i=1,...,l$, $H^i (G, \pi)=0$.
\end{theorem} 

\begin{remark}
In \cite{DymaraJ}[Theorem 7.1] the assumptions of the theorem do not include the decomposition $X_\eta = \bigoplus_{\tau \subseteq \eta} X^\tau$, but assumptions regrading the spectral gap in the $1$-dimensional links from which this decomposition is deduced. However, the proof of the theorem only relies on the above decomposition, therefore the theorem can be stated as above. Also, \cite{DymaraJ}[Theorems 5.2, 7.1] are stated for continuous unitary representations on Hilbert spaces, but the proof of \cite{DymaraJ}[Theorem 7.1] and the proof of \cite{DymaraJ}[Theorem 5.2] based on \cite{DymaraJ}[Theorem 7.1] pass verbatim to continuous representations on Banach spaces.
\end{remark}

We would like to add an additional condition on $\Sigma$ that will be denoted ($\mathcal{B}_{\delta, r}$) (replacing the condition ($\mathcal{B}_\delta$) appearing in \cite{DymaraJ}):
\begin{enumerate}[label=($\mathcal{B}_{\delta, r}$)]
\item For every $\eta \in \Sigma(n-2)$, the link of $\eta$, denoted $\Sigma_\eta$, is finite bipartite graph with sides $V_{\eta,1}, V_{\eta,2}$. For any $\eta \in \Sigma(n-2)$ denote $$V_{min} (\eta) = \min \lbrace \vert V_{\eta,1} \vert,\vert V_{\eta,2} \vert\rbrace,$$ 
and denote $\kappa (\eta)$ to be the smallest positive eigenvalue of the normalized Laplacian of $\Sigma_\eta$, then
$$(1-\kappa (\eta) ) \left( V_{min} (\eta) \right)^{\frac{1}{r}} \leq \delta.$$
\end{enumerate}

\begin{remark}
We note that if condition $(\mathcal{B} 4)$ is fulfilled and if the $1$-dimensional links of $\Sigma$ are finite, then every $1$-dimensional link has to be a bipartite graph.
\end{remark}

The main source of examples of $(\Sigma, G)$ fulfilling $(\mathcal{B}1)-(\mathcal{B} 4)$ and $(\mathcal{B}_{\delta, r})$ are groups coming from BN-pairs (a reader unfamilier with the definition of a BN-pair can find it in \cite{BuildingsBook}[Chapter 6]), when $G$ is the group and $\Sigma$ is the building on which it acts. In \cite{DymaraJ} the following is proved:
\begin{proposition}\cite{DymaraJ}[Propositions 1.6,1.7] 
Let $G$ be a group coming from a BN-pair and let $\Sigma$ be the building on which it acts. Assume further that $\Sigma$ is non compact and has finite thickness. Then conditions $(\mathcal{B}1)-(\mathcal{B} 4)$ are fulfilled for $(\Sigma, G)$ and $\Sigma_D$ is contractible.
\end{proposition}

In order to check the condition $(\mathcal{B}_{\delta, r})$ in buildings, we recall that if a building $\Sigma$ has finite $1$-dimensional links, then these links are spherical building, i.e., they are thick generalized $m$-gons with $m=2,3,4,6,8$ (a reader unfamilier with generalized $m$-gons can find a good introduction in \cite{VanM}[Chapter 1]). 

\begin{proposition}
\label{condition B-delta,r for buildings}
Let $\Sigma$ be a building such that the $1$-dimensional links of $\Sigma$ are compact. Let $m'$ be the smallest integer such that all the links of $1$-dimensional links of $\Sigma$ are generalized $m$-gons with $m \leq m'$. Then for every 
$$r > 
\begin{cases}
4 & m' =3 \\
8 & m' =4 \\
18 & m' = 6 \\
20 & m' =8
\end{cases},$$ 
and every $\delta >0$, if the thickness of the building is large enough, then $( \mathcal{B}_{\delta, r} )$ holds for $\Sigma$.
\end{proposition}

\begin{proof}
Let $(V,E)$ be a generalized $m$-gon of order $(s,t)$ and assume without loss of generality that $s \geq t$. Denote $\kappa$ to be the smallest positive eigenvalue of the normalized Laplacian on $(V,E)$. If $m=2$, then $1-\kappa=0$ and therefore this case is of no interest to us.

For $m >2$ the spectral gap $\kappa$ was explicitly computed by Feit and Higman \cite{FeitHig} for all generalized $m$-gons (the reader can find a summation of these results in \cite{BallmannS}[Section 3]). 
We will not recall the exact values of $\kappa$ depending on $(s,t)$, but only the asymptotic behaviour of $1-\kappa$ as $s$ and $t$ tends to $\infty$:
$$1- \kappa \sim \begin{cases}
O(\frac{1}{\sqrt{t}}) = O(\frac{1}{\sqrt{s}}) & m =3 \\
O(\frac{1}{\sqrt{t}} + \frac{1}{\sqrt{s}}) & m=4,6,8
\end{cases}.$$

We recall that generalized $m$-gons are always bipartite graphs. Denote $V_1, V_2$ to be the vertices in the two sides of $(V,E)$ and denote
$$V_\min = \min \lbrace \vert V_1 \vert, \vert V_2 \vert \rbrace.$$ 
The exact value of $\vert V_\min \vert$ depending on $(s,t)$ is computed in \cite{VanM}[Corollary 1.5.5] (recall we assumed that $s \geq t$):
$$V_\min = 
\begin{cases}
t^2 + t +1 & m=3 \\
(st+1)(t+1) & m =4\\
((st)^2 + st +1))(t+1) & m =6\\
((st)^2+1)(st+1))(t+1) & m =8\\
\end{cases}.$$

In order to complete the proof, we will also need the following connections between $s$ and $t$ (see \cite{VanM}[Theorem 1.7.2]):
$$\begin{cases}
s=t & m=3 \\
t^{\frac{1}{2}} \leq s \leq t^2 & m=4,8 \\
t^{\frac{1}{3}} \leq s \leq t^3 & m=6
\end{cases}.$$
To conclude the proof, we combine all of the above in order to show that for $r$ as above, $(1-\kappa)\vert V_\min \vert^{\frac{1}{r}}$ tends to $0$ as $t$ tends to infinity.

\begin{dmath*}
(1-\kappa)\vert V_\min \vert^{\frac{1}{r}} \sim \begin{cases}
(t^2+t+1)^{\frac{1}{r}} \frac{1}{\sqrt{t}} & m =3 \\
(st+1)^{\frac{1}{r}} (t+1)^{\frac{1}{r}} (\frac{1}{\sqrt{t}} + \frac{1}{\sqrt{s}}) & m =4 \\
((st)^2 + st +1))^{\frac{1}{r}} (t+1)^{\frac{1}{r}} (\frac{1}{\sqrt{t}} + \frac{1}{\sqrt{s}}) & m=6 \\
((st)^2+1)(st+1))^{\frac{1}{r}} (t+1)^{\frac{1}{r}} (\frac{1}{\sqrt{t}} + \frac{1}{\sqrt{s}})  & m =8
\end{cases} \leq \\
\begin{cases}
(t^2+t+1)^{\frac{1}{r}} \frac{1}{\sqrt{t}} & m =3 \\
(t^2 t+1)^{\frac{1}{r}} (t+1)^{\frac{1}{r}} (\frac{1}{\sqrt{t}} + \frac{1}{\sqrt{t}}) & m =4 \\
((t^3 t)^2 + t^3 t +1))^{\frac{1}{r}} (t+1)^{\frac{1}{r}} (\frac{1}{\sqrt{t}} + \frac{1}{\sqrt{t}}) & m=6 \\
((t^2 t)^2+1)(t^2 t+1))^{\frac{1}{r}} (t+1)^{\frac{1}{r}} (\frac{1}{\sqrt{t}} + \frac{1}{\sqrt{t}})  & m =8
\end{cases} \leq
\begin{cases}
3^{\frac{1}{r}} t^{\frac{2}{r}} \frac{1}{\sqrt{t}} & m =3 \\
4^{\frac{1}{r}} t^{\frac{4}{r}} \frac{2}{\sqrt{t}} & m =4 \\ 
6^{\frac{1}{r}} t^{\frac{9}{r}} \frac{2}{\sqrt{t}} & m=6 \\
8^{\frac{1}{r}} t^{\frac{10}{r}} \frac{2}{\sqrt{t}} & m=8 
\end{cases},
\end{dmath*}
and the conclusion follows.
\end{proof}

\subsection{Averaged projections in a Banach space}  
Let $X$ be a Banach space. Recall that a projection $P$ is a bounded operator $P \in \mathcal{B} (X)$ such $P^2 =P$.  Note that $\Vert P \Vert \geq 1$ if $P \neq 0$. For subspaces $M, N$ of $X$, we'll say that $P$ is a projection on $M$ along $N$ if $P$ is a projection such that $\im (P) = M$, $\Ker(P)=N$. 

Given a family of projections $P_1,...,P_N$ on  $M_1,...,M_N$ in $X$, there is a well known algorithm of finding a projection on $\cap_{j=1}^N M_j$, which is known as the method of averaged projections. The idea is to define the operator $T= \frac{P_1 +...+P_N}{N}$ and to take a limit $T^i$ as $i$ goes to infinity. The reader should note that in general $T^i$ need not converge in the operator norm. In \cite{ORobust}, the author had established a criterion for the convergence of $T^i$ using the idea of an angle between projections.

\begin{definition}[Angle between projections]
\label{angle between projections definition}
Let $X$ be a Banach space and let $P_1, P_2$ be projections on $M_1,M_2$ respectively. Assume that there is a projection $P_{1,2}$ on $M_1 \cap M_2$ such that $P_{1,2} P_1 = P_{1,2}$ and $P_{1,2} P_2 = P_{1,2}$ and define 
$$\cos (\angle (P_1,P_2)) = \max \left\lbrace \Vert P_1 (P_2 - P_{1,2} ) \Vert, \Vert P_2 (P_1 - P_{1,2} ) \Vert  \right\rbrace.$$ 
\end{definition}

\begin{remark}
In the above definition, we are actually defining the ``cosine'' of the angle. This is a little misleading, because we do not know if $\cos (\angle (P_1,P_2)) \leq 1$ holds in general (although this inequality holds in all the examples we can compute or bound). 
\end{remark}

\begin{remark}
We note that in the case where $X$ is a Hilbert space and $P_1,P_2$ are orthogonal projections on $M_1,M_2$, the orthogonal projection $P_{1,2}$ on  $M_1 \cap M_2$ will always fulfill $P_{1,2} P_1 = P_{1,2}$ and $P_{1,2} P_2 = P_{1,2}$. Also, in this case, $\cos (\angle (P_1,P_2))$ will be equal to the Friedrichs angle between $M_1$ and $M_2$ defined as
$$\cos (\angle (M_1,M_2))= \sup \lbrace \vert \langle u,v \rangle \vert : \Vert u \Vert \leq 1, \Vert v \Vert \leq 1, u \in M_1 \cap (M_1 \cap M_2)^\perp, v \in M_2 \rbrace.$$
\end{remark} 

Next, we recall the following theorems from \cite{ORobust}:
\begin{theorem} \cite{ORobust}[Theorem 3.12]
\label{Quick uniform convergence criterion}
Let $X$ be a Banach space and let $P_1,...,P_N$ be projections in $X$ ($N \geq 2$). Assume that for every $1 \leq j_1 < j_2 \leq N$, there is a projection $P_{j_1,j_2}$ on $\im (P_{j_1}) \cap \im (P_{j_2})$, such that $P_{j_1,j_2} P_{j_1} = P_{j_1,j_2}, P_{j_1,j_2} P_{j_2} = P_{j_1,j_2}$. 

Denote $T=\frac{P_1+...+P_N}{N}$ and assume there are constants 
$$\gamma < \frac{1}{8N-11} \text{ and } \beta < 1+ \frac{1-(8N-11)\gamma}{N-2 + (3N-4)\gamma},$$
such that  
$$ \max \lbrace \Vert P_1 \Vert,..., \Vert P_N \Vert \rbrace \leq \beta \text{ and } \max \lbrace  \cos(\angle (P_{j_1},P_{j_2})) : 1 \leq j_1 < j_2 \leq N \rbrace \leq \gamma.$$
Then for 
$$r = \dfrac{1+(N-2)\beta}{N}+(4-\dfrac{6}{N})\dfrac{1+\beta}{1-\gamma} \gamma,$$
$$C = \dfrac{(2N-2)\beta^2}{N (1-r)},$$
we have that $r <1$ and there is an operator $T^\infty$, such that $\Vert T^\infty - T^i  \Vert \leq C r^{i-1}$. Moreover, $T^\infty$ is a projection on $\bigcap_{j=1}^N \im (P_j)$.
\end{theorem}

To avoid carrying messy constants, we note the following:
\begin{corollary}
\label{Quick uniform convergence corollary}
In the notations of the above theorem, there are $\gamma_0 >0$ and $\beta_0 >1$ such that if 
$$ \max \lbrace \Vert P_1 \Vert,..., \Vert P_N \Vert \rbrace \leq \beta_0 \text{ and } \max \lbrace  \cos(\angle (P_{j_1},P_{j_2})) : 1 \leq j_1 < j_2 \leq N \rbrace \leq \gamma_0,$$
then $\Vert T^\infty - T^i \Vert \leq (4N) \left( \frac{2N-1}{2N} \right)^{i-1}$.
\end{corollary}

\begin{proof}
Note that in Theorem \ref{Quick uniform convergence criterion} above, $r$ tends to $\frac{1+(N-2)\beta}{N}$ as $\gamma$ tends to $0$. Therefore, we can choose $\beta_0 > 1$ and $\gamma_0$ small enough such that $r \leq \frac{2N-1}{2N}$. Also note that for such $r$, we have that $C=\frac{(2N-2)\beta_0^2}{N \frac{1}{2N}} = (4N-4) \beta_0^2$. Therefore, we can choose $\beta_0 >1$ small enough such that $C \leq 4N$.  
\end{proof}

Last, we note that $T^i$ converges to a ``canonical'' projection with respect to $P_1,...,P_N$ if such projection exists.

\begin{proposition}
\label{canonical proposition}
Let $X$ be a Banach space and let $P_1,...,P_N$ be projections in $X$ ($N \geq 2$). Denote $T=\frac{P_1+...+P_N}{N}$ and assume that $T^i$ converges in the operator norm to $T^\infty$ which is a projection on $\bigcap_{j=1}^N \im (P_j)$. If there is a projection $P_{1,2,...,N}$ on $\bigcap_{j=1}^N \im (P_j)$ such that for every $i$, $P_{1,2,...,N} P_j = P_{1,2,...,N}$, then $T^\infty = P_{1,2,...,N}$.
\end{proposition}

\begin{proof}
Note that for every $i$, we have that $P_{1,...,N} T^i = P_{1,...,N}$ and therefore $T^\infty = P_{1,...,N} T^\infty = P_{1,...,N}$.
\end{proof}

\subsection{Type and cotype}
Let $X$ be a Banach space. For $1< p_1 \leq 2$, $X$ is said to have (Gaussian) type $p_1$, if there is a constant $T_{p_1}$, such that for $g_1,...,g_n$ independent standard Gaussian random variables on a probability space $(\Omega, P)$, we have that for every $x_1,...,x_n \in X$ the following holds:
$$\left( \int_0^1 \left\Vert \sum_{i=1} g_i (\omega) x_i  \right\Vert^2 dP \right)^{\frac{1}{2}} \leq T_{p_1} \left( \sum_{i=1}^n \Vert x_i \Vert^{p_1} \right)^{\frac{1}{p_1}}.$$
The minimal constant $T_{p_1}$ such that this inequality is fulfilled is denoted $T_{p_1} (X)$. 

For $2 \leq p_2 < \infty$, $X$ is said to have (Gaussian) cotype $p_2$, if there is a constant $C_{p_2}$, such that for $g_1,...,g_n$ independent standard Gaussian random variables on a probability space $(\Omega, P)$, we have that for every $x_1,...,x_n \in X$ the following holds:
$$C_{p_2} \left( \sum_{i=1}^n \Vert x_i \Vert^{p_2} \right)^{\frac{1}{p_2}} \leq \left( \int_0^1 \left\Vert \sum_{i=1} g_i (\omega) x_i  \right\Vert^2 dP \right)^{\frac{1}{2}}  .$$
The minimal constant $C_{p_2}$ such that this inequality is fulfilled is denoted $C_{p_2} (X)$. 

We recall the following fact mentioned in the introduction regarding Banach spaces with given type and cotype which is due to Tomczak-Jaegermann \cite{Tomczak-Jaegermann}[Theorem 2 and the corollary after it]: if $X$ is a Banach space of type $p_1$, cotype $p_2$ and corresponding constants $T_{p_1} (X)$, $C_{p_2} (X)$ as above, then $d_k (X) \leq 4 T_{p_1} (X) C_{p_2} (X) k^{\frac{1}{p_1} - \frac{1}{p_2}}$. 

\begin{remark}
We remark that the Gaussian type and cotype defined above are equivalent to the usual (Rademacher) type and cotype (see \cite{HistoryBanach}[pages 311-312] and reference therein). 
\end{remark}

\begin{remark}
In \cite{PisierXu}, Pisier and Xu showed that for any $p_2 >2$ one can construct a non superreflexive Banach space $X$ with type $2$ and cotype $p_2$.
\end{remark}

\subsection{Vector valued $L^2$ spaces} 
\label{Vector valued spaces section}
Given a measure space $(\Omega, \mu)$ and Banach space $X$, a function $s : \Omega \rightarrow X$ is called simple if it is of the form:
$$s(\omega) = \sum_{i=1}^n \chi_{E_i} (\omega) v_i,$$
where $\lbrace E_1,...,E_n \rbrace$ is a partition of $\Omega$ where each $E_i$ is a measurable set, $\chi_{E_i}$ is the indicator function on $E_i$ and $v_i \in X$. 

A function $f : \Omega \rightarrow X$ is called Bochner measurable if it is almost everywhere the limit of simple functions. Denote $L^2 (\Omega ; X)$ to be the space of Bochner measurable functions such that 
$$\forall f \in L^2 (\Omega ; X), \Vert f \Vert_{L^2 (\Omega ; X)} = \left( \int_\Omega \Vert f (\omega) \Vert^2_X d \mu (\omega) \right)^{\frac{1}{2}} < \infty.$$ 

Given an operator $T \in B(L^2 (\Omega, \mu))$, we can define $T \otimes id_X \in B(L^2 (\Omega ; X))$ by defining it first on simple functions. 

For our uses, it will be important to bound the norm of an operator of the form $T \otimes id_X$ given that $X$ is derived by one of the deformation procedures given in the introduction. 

We will start by bounding the norm of $T \otimes id_X$ given that $X$ has ``round'' enough finite dimensional subspaces. For this, following \cite{LaatSalle2}, we introduce the following notation: for a Banach space $X$ and a constant $k \in \mathbb{N}$ denote 
$$e_k (X) = \sup \lbrace \Vert T \otimes id_X \Vert_{L^2 (\Omega ; X)} : T \text{ is of rank } k \text{ with } \Vert T \Vert_{\ell_2} \leq 1 \rbrace.$$

By a theorem of Pisier (see \cite{LaatSalle2}[Theorem 5.2]), this constant is connected to the constant $d_k (X)$ defined in the introduction by the inequality $e_k (X) \leq 2 d_k (X)$   
(there is also a reverse inequality $d_k (X) \leq e_k (X)$ which we will not use). Next, we recall the following definition:

\begin{definition}
For a Hilbert space $H$ and a bounded operator $T \in B(H)$ and a constant $r \in [1,\infty]$, the $r$-th Schatten norm is defined as 
$$r < \infty, \Vert T \Vert_{S^r} = \left( \sum_{i=1}^\infty (s_i (T))^r \right)^{\frac{1}{r}},$$
$$\Vert T \Vert_{S^\infty} =  s_1 (T),$$ 
where $s_1 (T) \geq s_2 (T) \geq ...$ are the eigenvalues of $\sqrt{T^* T}$. An operator $T$ is said to be of Schatten class $r$ if $\Vert T \Vert_{S^r} < \infty$. 
\end{definition}

In \cite{Salle} the following connection was between $e_k (X)$ and the norm of $T \otimes id_X$: 
\begin{lemma} \cite{Salle}[Proposition 3.3]
Let $r \in [2, \infty)$, $r >r' \geq 2$ be constants and assume there is a constant $C'$ such that $e_k (X) \leq C' k^{\frac{1}{r}}$ for every $k$. Denote
$$M = \sum_{i=1}^\infty 2^{\frac{r'}{r'-1} (\frac{1}{r} - \frac{1}{r'} )i}  .$$
If $(\Omega, \mu)$ is a measure space and  $T \in B(L^2 (\Omega,\mu))$ is of Schatten class $r'$, then
$$\Vert T \otimes id_X \Vert_{B(L^2(\Omega ; X))} \leq M C' \Vert T \Vert_{S^{r'}}.$$ 
\end{lemma}

\begin{remark}
The statement of \cite{Salle}[Proposition 3.3] refers to Banach spaces with specified type and cotype, but is only uses the fact that for these spaces $e_k (X)$ can be bounded by some $ C' k^{\frac{1}{r}}$. Therefore the proof of \cite{Salle}[Proposition 3.3] actually prove the more general case stated above (this was already observed and used in \cite{LaatSalle2}).  
\end{remark}

Combining the above lemma with the theorem of Pisier stated above gives the following corollary:

\begin{corollary}
\label{norm of T otimes id using Schatten} 
Let $r \in [2, \infty)$, $r >r'$ be constants and assume there is a constant $C_1$ such that $d_k (X) \leq C_1 k^{\frac{1}{r}}$ for every $k$. Then there is a constant $C=C(C_1,r,r')$ such that for every measure space $(\Omega, \mu)$ and every $T \in B(L^2 (\Omega,\mu))$ of Schatten class $r'$, we have that
$$\Vert T \otimes id_X \Vert_{B(L^2(\Omega ; X))} \leq C \Vert T \Vert_{S^{r'}}.$$ 
\end{corollary}

Second, we will see that if $X$ is given as an interpolation of two spaces $X_0$, $X_1$, the norm of $T \otimes id_X$ can be bounded using bounds on the norms of $T \otimes id_{X_0}, T \otimes id_{X_1}$:

\begin{lemma} \cite{Salle}[Lemma 3.1]
\label{interpolation fact}
Given a compatible pair $(X_0,X_1)$, a measure space $(\Omega,\mu)$ and an operator $T \in B(L^2 (\Omega,\mu))$, we have for every $0 \leq \theta \leq 1$ that
$$\Vert T \otimes id_{[X_0, X_1]_\theta} \Vert_{B(L^2 (\Omega ; [X_0, X_1]_\theta))} \leq \Vert T \otimes id_{X_0} \Vert_{B(L^2 (\Omega ; X_0))}^{1-\theta} \Vert T \otimes id_{X_1} \Vert_{B(L^2 (\Omega ; X_1))}^{\theta},$$
where $[X_0, X_1]_\theta$ is the interpolation of $X_0$ and $X_1$ (see definition above).
\end{lemma}

Third, if $X$ and $X'$ are isomorphic then the norm on $T \otimes id_X$ can be bounded using the norm on $T \otimes id_{X'}$ and the Banach-Mazur distance between $X$ and $X'$.

\begin{lemma} \cite{ORobust}[Lemma 2.7]
\label{norm of T otimes id using BM}
Let $(\Omega, \mu)$ be a measure space and $T$ a bounded operator on $L^2 (\Omega, \mu)$. Given two isomorphic Banach spaces $X$, $X'$, we have that 
$$\Vert T \otimes id_X \Vert_{B(L^2(\Omega ; X))} \leq d_{BM} (X,X') \Vert T \otimes id_{X'} \Vert.$$ 
\end{lemma}

Last, we need the following fact of regarding passing to the closure under quotients, subspaces, $l_2$-sums and ultraproducts:

\begin{lemma}\cite{Salle}[Lemma 3.1]
\label{L2 norm stability}
Let $(\Omega, \mu)$ be a measure space, $C\geq 0$ and $T$ a bounded operator on $L^2 (\Omega, \mu)$. The class of Banach spaces $X$, for which $\Vert T \otimes id_X \Vert \leq C$ is stable under quotients, subspaces, $l_2$-sums and ultraproducts.
\end{lemma}

\begin{remark}
The fact that the above class is closed under $l_2$ sums, did not appear in \cite{Salle}[Lemma 3.1] and it is left as an exercise to the reader.
\end{remark}

Combining all the results above yields the following:
\begin{corollary}
\label{bounding norm of composition of deformations}
Let $T \in B(L^2 (\Omega,\mu))$ be an operator and let $L \geq 1, r' \geq 2$ be constants such that $\Vert T \Vert_{S^{r'}} \leq 1$ and such that for every Banach space $X$ we have that $\Vert T \otimes id_X \Vert_{B(L^2(\Omega ; X))}  \leq L$. Then for every constants $r>r', C_1 \geq 1, 1 \geq  \theta_2 > 0, C_3 \geq 1$, there is a constant $C=C(C_1,r,r')$ such that for every $X \in \overline{\mathcal{E}_3 (\mathcal{E}_2 (\mathcal{E}_1 (r, C_1),\theta_2),C_3)}$ the following holds
$$\Vert T \otimes id_X \Vert_{B(L^2(\Omega ; X))}  \leq C_3 L (C \Vert T \Vert_{S^{r'}})^{\theta_2} .$$
\end{corollary}

\begin{proof}
By Corollary \ref{norm of T otimes id using Schatten} there is a constant $C=C(C_1,r,r')$ such that for every $X \in \mathcal{E}_1 (r, C_1)$ the following holds:
$$\Vert T \otimes id_X \Vert_{B(L^2(\Omega ; X))} \leq C \Vert T \Vert_{S^{r'}}.$$
Combining this with Lemma \ref{interpolation fact} and our assumptions on $T$ gives that for every $X \in \mathcal{E}_2 (\mathcal{E}_1 (r, C_1),\theta_2)$, we have that
$$\Vert T \otimes id_X \Vert_{B(L^2(\Omega ; X))} \leq L (C \Vert T \Vert_{S^{r'}})^{\theta_2}.$$
Applying Lemma \ref{norm of T otimes id using BM} yields that for every $X \in \mathcal{E}_3 (\mathcal{E}_2 (\mathcal{E}_1 (r, C_1),\theta_2),C_3)$, we have that
$$\Vert T \otimes id_X \Vert_{B(L^2(\Omega ; X))}  \leq C_3 L^{1-\theta_2} (C \Vert T \Vert_{S^{r'}})^{\theta_2} .$$
Last, Lemma \ref{L2 norm stability} states that this inequality does not change when passing to the closure.
\end{proof}

\subsection{Group representations in a Banach space} 
Let $G$ be a locally compact group and $X$ a Banach space. Let $\pi$ be a representation $\pi :  G \rightarrow \mathcal{B} (X)$. Throughout this paper we shall always assume $\pi$ is continuous with respect to the strong operator topology without explicitly mentioning it. 

Denote by $C_c (G)$ the groups algebra of compactly supported simple functions on $G$ with convolution. For any $f \in C_c (G)$ we can define $\pi (f) \in \mathcal{B} (X)$ as
$$\forall v \in X, \pi (f). v = \int_G f(g) \pi(g).v  d\mu (g),$$
where the above integral is the Bochner integral with respect to the (left) Haar measure $\mu$ of $G$. 

Recall that given $\pi$ one can define the following representations:
\begin{enumerate}
\item The complex conjugation of $\pi$, denoted $\overline{\pi} : G \rightarrow \mathcal{B} (\overline{X})$ is defined as 
$$\overline{\pi} (g). \overline{v} = \overline{\pi (g). v}, \forall g \in G, \overline{v} \in \overline{X}.$$  
\item The dual representation $\pi^* : G \rightarrow \mathcal{B} (X^*)$ is defined as 
$$\langle v, \pi^* (g) u  \rangle =  \langle \pi (g^{-1}) .v, u  \rangle, \forall g \in G, v \in X, u \in X^*.$$ 
\end{enumerate}

Next, we'll restrict ourselves to the case of compact groups. Let $K$ be a compact group with a Haar measure $\mu$ and let $C_c (K) = C(K)$ defined as above. Let $X$ be Banach space and let $\pi$ be a representation of $K$ on $X$ that is continuous with respect to the strong operator topology. We shall show that for every $f \in C_c (K)$, we can bound the norm of $\pi (f)$ using the norm of $\lambda \otimes id_X \in B(L^2 (K ;X))$ (the definition of $L^2 (K; X)$ is given in subsection \ref{Vector valued spaces section} above).

\begin{proposition}\cite{ORobust}[Corollary 2.11]
\label{bounding the norm of pi(f) - proposition}
Let $\pi$ be a representation of a compact group $K$ on a Banach space $X$. Then for any real function $f \in C_c (G)$ we have that
$$\Vert \pi (f) \Vert_{B(X)} \leq \left( \sup_{g \in K} \Vert \pi (g) \Vert \right)^2  \Vert (\lambda \otimes id_X) (f) \Vert_{B(L^2 (K ; X))},$$
where $\lambda$ is the left regular representation of $G$. 
\end{proposition}

\subsection{Asplund spaces}

\begin{definition}
A Banach space $X$ is said to be an Asplund space if every separable subspace of $X$ has a separable dual.
\end{definition}

There are many examples of Asplund spaces - for instance every reflexive space is Asplund. A very nice exposition of Asplund spaces was given by Yost in \cite{Yost}. The reason we are interested in Asplund space is the following theorem of Megrelishvili:

\begin{theorem}\cite{Megre}[Corollary 6.9]
\label{Asplund implies continuous dual rep}
Let $G$ be a topological group and let $\pi$ be a continuous representation of $G$ on a Banach space $X$. If $X$ is an Asplund space, then the dual representation $\pi^*$ is also continuous.
\end{theorem}

\section{Angle between more than 2 projections and space decomposition}
\label{Angle between more than 2 projections and space decomposition}
The aim of this section is to show that given several projections on a Banach space, this space can be decomposed with respect to these projections, given that the angle between every two projections is large enough. The main motivation for establishing such a decomposition is applying it to deduce vanishing of cohomology relying on Theorem \ref{general vanishing of cohomology based on decomposition}. In order to prove this decomposition, we define and study the notion an angle between several projections. 

Following our main motivation, we will think about our projections as defined by faces of a simplex:
\begin{definition}
\label{simplex projections definition}
Let $X$ be a Banach space and let $\triangle = \lbrace 0,...,n \rbrace$ be a simplex with $n+1$ vertices. For $k = -1,0,...,n$, denote by $\triangle (k)$ the $k$-dimensional faces of $\triangle$, i.e., the subsets of $\triangle$ with cardinality $k+1$.

Let $P_\sigma$ be projections defined for every $\sigma \in \triangle (n) \cup \triangle (n-1)$ such that
$$\forall \sigma \in \triangle (n-1), P_\sigma P_\triangle = P_\sigma.$$

For every $\tau \subseteq \triangle$ define an operator $T_\sigma$ as follows:
$$T_\tau = \begin{cases}
P_\triangle & \tau = \triangle \\
\dfrac{\sum_{\sigma \in \triangle  (n-1), \tau \subseteq \sigma  } P_\sigma}{\vert \triangle \setminus \tau \vert} & \tau \neq \triangle
\end{cases}
.$$

Fix $\tau \subsetneqq \triangle$. If $T_\tau^i$ converges to a projection on the space $\cap_{\sigma \in \triangle  (n-1), \tau \subseteq \sigma} \im (P_\sigma)$ as $i \rightarrow \infty$, then we define $P_\tau = \lim T_\tau^i$. In this case we say that $P_\tau$ exists.
\end{definition}

\begin{remark}
We note that the above setting is general for any $n+1$ projections $P_0,...,P_{n}$. Indeed, given any such projections, we can always denote $P_{i} = P_{\triangle \setminus \lbrace i \rbrace}$ and take $P_\triangle = I$ (the reason we define the operator $P_\triangle$ above is that in the setting we will consider, such an operator appears naturally). 
\end{remark}

\begin{remark}
\label{P-triangle always commute remark}
By the definition of $P_\triangle$, we have for every $\tau \subseteq \triangle$ and every $i$ that
$$T_\tau^i P_\triangle = T_\tau^i P_\triangle = T_\tau^i.$$
Therefore for every $\tau \subseteq \triangle$, if $P_\tau$ exists, then $P_\tau P_\triangle = P_\tau$. 
\end{remark}

Using this notations, we will define the $\cos$ of an angle between more than $2$ projections:

\begin{definition}
Let $X$ and $P_\sigma$ for $\sigma \in \triangle (n-1)$ be defined as in definition \ref{simplex projections definition} above.
Fix $1 \leq k \leq n$. Denote $Sym ( 0,1,...,k )$ to be the group of all permutations of $\lbrace 0,1,...,k \rbrace$. 

For $\sigma_0,...,\sigma_k \in \triangle (n-1)$ pairwise disjoint, denote $\tau = \bigcap_{i=0}^k \sigma_i$. If $P_\tau$ exists, define $\cos (\angle (P_{\sigma_0},...,P_{\sigma_k}))$ as 
$$\cos (\angle (P_{\sigma_0},...,P_{\sigma_k})) = \max_{\pi \in Sym (0,...,k)} \Vert P_{\sigma_{\pi (0)}} P_{\sigma_{\pi (1)}} ... P_{\sigma_{\pi (k)}} (I- P_\tau ) \Vert.$$
\end{definition}


\begin{theorem}
\label{angle between several projections theorem}
Let $X$, $\triangle$, $P_\sigma$ for $\sigma \in \triangle (n-1)$ be defined as above and assume $n>1$. Assume that for every $\eta \in \triangle (n-2)$, the projection $P_{\eta}$ exists and that
$$\forall \sigma \in \triangle (n-1), \eta \subseteq \sigma \Rightarrow P_\eta P_\sigma = P_\eta.$$
Also assume that $ \max_{\sigma \in \triangle (n-1)}  \Vert P_\sigma \Vert \leq \beta_0,$
where $\beta_0 >1$ is the constant of Corollary \ref{Quick uniform convergence corollary}. 

Then for every $\varepsilon >0$ there is $\gamma>0$ such that if 
$$\max \lbrace  \cos(\angle (P_\sigma,P_{\sigma '})) : \sigma, \sigma' \in \triangle (n-1) \rbrace \leq \gamma.$$
then for every $\tau \subseteq \triangle$, $P_\tau$ is well defined and for every pairwise disjoint $\sigma_0,...,\sigma_k \in \triangle (n-1)$ the following holds:
$$\cos (\angle (P_{\sigma_0},...,P_{\sigma_k})) \leq \varepsilon.$$
\end{theorem}

\begin{proof}
Let $\gamma_0>0$ and $\beta_0$ be the constants of Corollary \ref{Quick uniform convergence corollary} and fix $\varepsilon >0$. Note that $\beta_0 \leq 2$.

Fix $1 \leq k \leq n$ and $\sigma_0,...,\sigma_k \in \triangle (n-1)$. Denote $\tau = \cap_{j=0}^k \sigma_j$. 

Assume first that $\gamma \leq \gamma_0$, then by Corollary \ref{Quick uniform convergence corollary}, we have that $T_\tau^i$ converges to $P_\tau$ and 
\begin{equation}
\label{rate of conv ineq}
\Vert P_\tau - T_\tau^i \Vert \leq 4(k+1) \left( \frac{2(k+1)-1}{2(k+1)} \right)^{i-1} \leq 4(n+1) \left( \frac{2(n+1)-1}{2(n+1)} \right)^{i-1}. 
\end{equation}

Without loss of generality, it is enough to show that there is $\gamma$ such that 
$$\Vert P_{\sigma_0} ... P_{\sigma_k} (I- P_\tau ) \Vert \leq \varepsilon.$$
By \eqref{rate of conv ineq}, we can choose $i_0$ large enough such that 
$$\Vert P_\tau - T_\tau^{i_0} \Vert \leq \frac{\varepsilon}{2^{n+2}} ,$$ 
and this $i_0$ can be chosen independently of $k$.

Therefore 
\begin{align*}
\Vert  P_{\sigma_0} ... P_{\sigma_k} (I- P_\tau ) \Vert \leq \Vert  P_{\sigma_0} ... P_{\sigma_k} (I- T_\tau^{i_0} ) \Vert + \Vert  P_{\sigma_0} ... P_{\sigma_k} (T_\tau^{i_0}- P_\tau ) \Vert \leq \\
\Vert  P_{\sigma_0} ... P_{\sigma_k} (I- T_\tau^{i_0} ) \Vert + \Vert P_{\sigma_0} \Vert ... \Vert P_{\sigma_k} \Vert \frac{\varepsilon}{2^{n+2}} \leq \\
\Vert  P_{\sigma_0} ... P_{\sigma_k} (I- T_\tau^{i_0} ) \Vert + \beta_0^{k+1} \frac{\varepsilon}{2^{n+2}} \leq \\ \Vert  P_{\sigma_0} ... P_{\sigma_k} (I- T_\tau^{i_0} ) \Vert+ \frac{\varepsilon}{2}.
\end{align*}
We are left to show that by choosing $\gamma$ small enough, we can ensure that 
$$\Vert  P_{\sigma_0} ... P_{\sigma_k} (I- T_\tau^{i_0} ) \Vert \leq \dfrac{\varepsilon}{2}.$$
Denote $T_\tau' = I- T_\tau= \frac{(I-P_{\sigma_0}) +...+(I-P_{\sigma_k})}{k+1}$. Note that
$$I-T_\tau^{i_0} =T_\tau ' \left({i_0 \choose 1} I - {i_0 \choose 2} T_\tau ' + ...+ (-1)^{i_0-1} {i_0 \choose i_0} \left( T_\tau ' \right)^{i_0-1} \right).$$
Recall that by our assumptions $\Vert T_\tau \Vert \leq \beta_0$ and therefore that $\Vert T_\tau ' \Vert \leq 1+\beta_0 \leq 3$. This yields that
\begin{dmath*}
\Vert  P_{\sigma_0} ... P_{\sigma_k} (I- T_\tau^{i_0} ) \Vert \leq \\ 
\Vert  P_{\sigma_0} ... P_{\sigma_k} T_\tau ' \Vert \left\Vert {i_0 \choose 1} I - {i_0 \choose 2} T_\tau ' + ...+ (-1)^{i_0-1} {i_0 \choose i_0} \left( T_\tau ' \right)^{i_0-1} \right\Vert \leq \\
\Vert   P_{\sigma_0} ... P_{\sigma_k} T_\tau ' \Vert \left(\Vert I \Vert +{i_0 \choose 2} \Vert T_\tau ' \Vert + ... +  {i_0 \choose i_0} \Vert T_\tau ' \Vert^{i_0-1} \right) \leq \\
\Vert   P_{\sigma_0} ... P_{\sigma_k} T_\tau ' \Vert \dfrac{1}{3}\left(3 +{i_0 \choose 2} 3^2 + ... +  {i_0 \choose i_0} 3^{i_0} \right) \leq \Vert   P_{\sigma_0} ... P_{\sigma_k} T_\tau ' \Vert \dfrac{4^{i_0}}{3}. 
\end{dmath*}
Therefore it is enough to show we can choose $\gamma$ small enough such that 
$$\Vert   P_{\sigma_0} ... P_{\sigma_k} T_\tau ' \Vert \leq \dfrac{3}{4^{i_0}} \dfrac{\varepsilon}{2},$$
(note that $i_0$ is independent of $\gamma$ as long as $\gamma \leq \gamma_0$). We will finish the proof by showing that 
\begin{equation}
\label{claimed ineq}
\Vert   P_{\sigma_0} ... P_{\sigma_k} T_\tau ' \Vert \leq n 2^{n+1} \gamma.
\end{equation}
By the definition of $ T_\tau '$, we have that
\begin{dmath*}
\Vert P_{\sigma_0} ... P_{\sigma_k} T_\tau ' \Vert \leq  
\left\Vert \dfrac{ P_{\sigma_0} ... P_{\sigma_k} (I-P_{\sigma_0})}{k+1} \right\Vert +...+  \left\Vert \dfrac{ P_{\sigma_0} ... P_{\sigma_k} (I-P_{\sigma_k})}{k+1} \right\Vert.
\end{dmath*}

Therefore, in order to prove inequality \eqref{claimed ineq}, it is enough to show that for every $j,k$ such that $k\geq j \geq 0$, we have that 
$$\Vert P_{\sigma_0} ... P_{\sigma_k} (I-P_{\sigma_j}) \Vert \leq (k-j) 2^{k+1} \gamma.$$
We will show this by induction on $k-j$. If $k-j=0$, i.e., if $k=j$ then 
$$P_{\sigma_0} ... P_{\sigma_k} (I-P_{\sigma_k}) =0,$$
and we are done. Assume that $k>j$ and that the inequality holds for $k-1,j$, i.e., assume that
$$\Vert P_{\sigma_0} ... P_{\sigma_{k-1}} (I-P_{\sigma_j}) \Vert \leq (k-1-j) 2^{k} \gamma.$$
Then for $k$ and $j$ we have that
\begin{dmath*}
P_{\sigma_0} ... P_{\sigma_{k}} (I-P_{\sigma_j}) = P_{\sigma_0} ... P_{\sigma_{k-1}} (P_{\sigma_{k}}-P_{\sigma_{k}} P_{\sigma_j}) = P_{\sigma_0} ... P_{\sigma_{k-1}} (P_{\sigma_{k}}- P_{\sigma_j} P_{\sigma_{k}}) +P_{\sigma_0} ... P_{\sigma_{k-1}} (P_{\sigma_j} P_{\sigma_{k}}-P_{\sigma_{k}} P_{\sigma_j})  = P_{\sigma_0} ... P_{\sigma_{k-1}} (I- P_{\sigma_j}) P_{\sigma_{k}} +P_{\sigma_0} ... P_{\sigma_{k-1}} (P_{\sigma_j} P_{\sigma_{k}}-P_{\sigma_{k}} P_{\sigma_j}).
\end{dmath*}
Therefore 
\begin{dmath*}
{\Vert P_{\sigma_0} ... P_{\sigma_{k}} (I-P_{\sigma_j}) \Vert} \leq \\ \Vert P_{\sigma_0} ... P_{\sigma_{k-1}} (I- P_{\sigma_j}) P_{\sigma_{k}} \Vert + \Vert P_{\sigma_0} ... P_{\sigma_{k-1}} (P_{\sigma_j} P_{\sigma_{k}}-P_{\sigma_{k}} P_{\sigma_j}) \Vert.
\end{dmath*}
Note that 
$$\Vert P_{\sigma_j} P_{\sigma_{k}}-P_{\sigma_{k}} P_{\sigma_j} \Vert \leq \Vert P_{\sigma_j} P_{\sigma_{k}}-P_{\sigma_{k} \cap \sigma_j} \Vert + \Vert P_{\sigma_k} P_{\sigma_{j}}-P_{\sigma_{k} \cap \sigma_j} \Vert \leq 2 \gamma,$$
and therefore
$$\Vert P_{\sigma_0} ... P_{\sigma_{k-1}} (P_{\sigma_j} P_{\sigma_{k}}-P_{\sigma_{k}} P_{\sigma_j}) \Vert \leq \Vert P_{\sigma_0} ... P_{\sigma_{k-1}} \Vert 2 \gamma \leq 2^{k+1} \gamma.$$
Also, note that by the induction assumption
$$\Vert P_{\sigma_0} ... P_{\sigma_{k-1}} (I- P_{\sigma_j}) P_{\sigma_{k}} \Vert \leq (k-1-j) 2^{k} \gamma \Vert P_{\sigma_{k}} \Vert \leq (k-1-j) 2^{k+1} \gamma.$$
Combining the two inequalities above yields
\begin{dmath*}
{\Vert P_{\sigma_0} ... P_{\sigma_{k}} (I-P_{\sigma_j}) \Vert} \leq (k-j) 2^{k+1} \gamma,
\end{dmath*}
as needed.
\end{proof}

\begin{definition} [Consistency]
\label{consistency definition}
Let $X$, $\triangle$, $P_\sigma$ for $\sigma \in \triangle (n-1)$ defined as above. We shall say that the projections $P_\sigma$ for $\sigma \in \triangle (n-1)$ are consistent, given that for every $\tau \subseteq \eta \subsetneqq \triangle$, if $P_\tau$ and $P_\eta$ exist then $P_\tau P_\eta = P_\tau$.
\end{definition}

\begin{remark}
If the projections $P_\sigma$ for $\sigma \in \triangle (n-1)$ are consistent and $P_\tau$ exists for every $\tau \subseteq \triangle$, then for every $\tau , \tau' \subseteq \triangle$, we can define $\cos ( \angle (P_\tau, P_{\tau '}))$ as in the background section, i.e.,
$$\cos ( \angle (P_\tau, P_{\tau '})) = \max \lbrace \Vert P_\tau P_{\tau '} - P_{\tau \cap \tau '} \Vert, \Vert P_{\tau '} P_\tau - P_{\tau \cap \tau '} \Vert \rbrace.$$
\end{remark}

\begin{proposition}
\label{consistency is checked on P_i's proposition}
Let $X$, $\triangle$, $P_\sigma$ for $\sigma \in \triangle (n-1)$ defined as above. Assume that for every $\tau \in \triangle$, $P_\tau$ exists. Then the projections $P_\sigma$ for $\sigma \in \triangle (n-1)$ are consistent if and only if for
$$\forall \tau \subsetneqq \triangle, \forall \sigma \in \triangle (n-1), \tau \subseteq \sigma \Rightarrow P_\tau P_\sigma = P_\tau.$$
\end{proposition}

\begin{proof}
One direction is trivial - assume that the projections $P_\sigma$ for $\sigma \in \triangle (n-1)$ are consistent, then for every $\tau \subseteq \eta \subsetneqq \triangle$, we have that $P_\tau P_\eta = P_\tau$ and in particular this holds for every $\eta \in \triangle (n-1)$.

In the other direction, fix some  $\tau \subseteq \eta \subsetneqq \triangle$. By our assumptions, we have for every $\sigma \in \triangle (n-1)$, $\tau \subseteq \sigma$ that $P_\tau P_\sigma = P_\tau$. Therefore, by the definition of $T_\eta$,
$$\forall i, P_\tau (T_\eta)^i = P_\tau,$$
which in turn implies that $P_\tau P_\eta = P_\tau$ as needed.
\end{proof}

\begin{proposition}
Let $X$, $\triangle$, $P_\sigma$ for $\sigma \in \triangle (n-1)$ defined as above. Assume that for every $\tau \subseteq \triangle$, $P_\tau$ exists. If for every $\tau \subsetneqq \triangle$ there is a projection $P_\tau '$ on $\cap_{\sigma \in \triangle (n-1), \tau \subseteq \sigma} \im (P_\sigma)$ such that 
$$\forall \sigma \in \triangle (n-1), \tau \subseteq \sigma  \Rightarrow P_\tau' P_\sigma = P_\tau ',$$
then the projections $P_\sigma$ for $\sigma \in \triangle (n-1)$ are consistent and for every $\tau \subsetneqq \triangle$, $P_\tau = P_\tau '$.
\end{proposition}

\begin{proof}
By Proposition \ref{canonical proposition}, we have that $T^i_\tau$ converges to $P_\tau '$ for every $\tau \subsetneqq \triangle$ and the consistency follows from Proposition \ref{consistency is checked on P_i's proposition}. 
\end{proof}

The main tool that we will use to decompose the space $X$ is the following theorem stating that bounding the angle between each $P_{\sigma},P_{\sigma'}$ where $\sigma, \sigma' \in \triangle (n-1)$ gives a bound on the angle between $P_\tau,P_{\tau '}$ where $\tau, \tau '$ are any faces of $\triangle$.
\begin{theorem}
\label{small angle theorem}
Let $X$, $\triangle$, $P_\triangle$ and $P_\sigma$ for $\sigma \in \triangle (n-1)$ defined as above.  Assume the following:
\begin{enumerate}
\item The projections $P_\sigma$ for $\sigma \in \triangle (n-1)$ are consistent.
\item For any $\eta \in \triangle (n-2)$, the projections $P_\eta$ exist.
\item $ \max_{\sigma \in \triangle (n-1) \cup \triangle (n)} \Vert P_\sigma \Vert \leq \beta_0,$ where $\beta_0 >1$ is the constant of Corollary \ref{Quick uniform convergence corollary}. 
\end{enumerate}
Then for every $\varepsilon >0$ there is $\gamma>0$ such that if 
$$ \max \lbrace  \cos(\angle (P_\sigma,P_{\sigma '})) : \sigma, \sigma ' \in \triangle (n-1) \rbrace \leq \gamma.$$
then the following holds:
\begin{enumerate}
\item For every $\tau \subseteq \triangle$, $P_\tau$ exists and $\Vert P_\tau \Vert \leq 4(n+1)+2$.
\item For every $\tau, \tau' \subseteq \triangle$ and every $\eta \subseteq \triangle$ such that $\tau \cap \tau' \subseteq \eta$ we have that
$$\Vert P_\tau P_{\tau'} (I-P_\eta) \Vert \leq \varepsilon.$$ 
In particular, $\cos (\angle (P_\tau,P_{\tau '})) \leq \varepsilon$.
\end{enumerate}

\end{theorem}

\begin{remark}
Variations of the above theorem were proven in the setting of Hilbert spaces in \cite{DymaraJ}, \cite{ErshovJZ} and \cite{Kassabov}. However, all these proofs use the fact that in a Hilbert space the following equality holds for any two subspaces $U,V$: $\angle (V,U) = \angle (V^\perp, U^\perp)$, where the angle here is the Friedrichs angle. In our setting, we do not know if such equality holds, namely if $\cos (\angle (P_\tau,P_{\tau '})) =  \cos (\angle (I-P_{\tau},I-P_{\tau'}))$ (we don't even know if $\cos (\angle (I-P_{\tau},I-P_{\tau'}))$ is well defined). This limitation required us to give a more direct proof using the idea of angle between several projections. 
\end{remark}

\begin{proof}
Let $\gamma_0>0$ and $\beta_0$ be the constants of Corollary \ref{Quick uniform convergence corollary} and let $\varepsilon ' >0$ be a constant to be determined later. By Theorem \ref{angle between several projections theorem}, there is a constant $\gamma_1 >0$ such that if 
$$ \max \lbrace  \cos(\angle (P_\sigma,P_{\sigma '})) : \sigma, \sigma ' \in \triangle (n-1) \rbrace \leq \gamma,$$
then for any $k = 1,...,n$ and for any $\eta \in \triangle (n-1-k)$, we have that
$$\cos (\angle (P_{\sigma_0},...,P_{\sigma_k} )) \leq \varepsilon',$$
where $\sigma_0,...,\sigma_k \in \triangle (n-1)$ are all the $n-1$ faces of $\triangle$ that contain $\eta$. Choose $\gamma = \min \lbrace \gamma_0, \gamma_1 \rbrace$.

If $\tau \in \triangle (n-1) \cup \triangle (n)$, then $P_\tau$ exists and $\Vert P_\tau \Vert \leq \beta_0 < 2 \leq 4(n+1)+2$. Assume next $\vert \tau \vert < n$, then by Corollary \ref{Quick uniform convergence corollary} we have that $P_\sigma$ is exists and 
$$\Vert P_\tau \Vert \leq 4(n+1) + \Vert T_\tau \Vert \leq 4(n+1)+\beta_0 \leq 4(n+1)+2.$$
This concludes the proof of the first assertion of the theorem.

Let $\tau, \tau ' \subseteq \triangle$ and $\eta \subseteq \triangle$ such that $\tau \cap \tau ' \subseteq \eta$. First, we note that by the consistency assumption $P_{\tau \cap \tau' } (I-P_\eta)=0$ and therefore
\begin{dmath*}
{\Vert P_\tau P_{\tau '} (I-P_\eta) \Vert = \Vert P_\tau P_{\tau '} (I-P_{\tau \cap \tau' })(I-P_\eta) \Vert \leq} \\ \cos (\angle (P_\tau, P_{\tau '})) \Vert I-P_\eta \Vert \leq (4(n+1)+3)\cos (\angle (P_\tau, P_{\tau '})).
\end{dmath*}
Therefore, it is enough to show that for $\gamma$ small enough
$$\Vert P_\tau P_{\tau '}(I - P_{\tau \cap \tau '}) \Vert \leq \dfrac{\varepsilon}{(4(n+1)+3)},$$
for any $\tau, \tau' \subseteq \triangle$. 

Note that if $\tau = \tau'$ or $\tau = \triangle$ or $\tau' = \triangle$, then $\cos (\angle (P_\tau, P_{\tau '})) =0$ and there is nothing to prove. Therefore, we can assume that $\tau \cap \tau' \in \triangle (n-1-k)$ for $1 \leq k \leq n$. Let $\sigma_0,...,\sigma_k \in \triangle (n-1)$ be all the pairwise disjoint simplices that contain $\tau \cap \tau'$. Without loss of generality we can assume that
$$\tau \subseteq \sigma_0,...,\tau \subseteq \sigma_j \text{ and } \tau' \subseteq \sigma_{j+1},...,\tau' \subseteq \sigma_k.$$
We note that by the consistency assumption
$$P_\tau = P_\tau P_{\sigma_0} ... P_{\sigma_j},$$
and
$$P_{\tau'} = P_{\sigma_{j+1}} ... P_{\sigma_k} P_{\tau'}.$$
Therefore
\begin{dmath*}
\Vert P_\tau P_{\tau '}(I - P_{\tau \cap \tau '}) \Vert = \\
 \Vert P_\tau  P_{\sigma_0} ... P_{\sigma_k} P_{\tau'} (I- P_{\tau \cap \tau '}) \Vert = \\
 \Vert P_\tau  P_{\sigma_0} ... P_{\sigma_k} (I- P_{\tau \cap \tau '}) P_{\tau'} \Vert \leq \\
  \Vert P_\tau \Vert \Vert P_{\tau'} \Vert \cos (\angle (P_{\sigma_0},...,P_{\sigma_k})) \leq (4(n+1)+2)^2 \varepsilon '. 
\end{dmath*}
We conclude by choosing $\varepsilon ' = \frac{\varepsilon}{(4(n+1)+2)^2 (4(n+1)+3)}$.
\end{proof}

\begin{remark}
Theorem \ref{small angle theorem} can be proven without the assumption that the projections $P_\sigma$ with $\sigma \in \triangle (n-1)$ are consistent. However, we could not prove Theorem \ref{space decomposition} below without this assumption (see remark after the proof of Theorem \ref{space decomposition}). Our motivation for proving Theorem \ref{small angle theorem} was deducing Theorem \ref{space decomposition} and therefore we assumed consistency in the proof (this assumption simplifies the proof considerably). For completeness, we added a proof of Theorem \ref{small angle theorem} in the appendix that does not rely on the consistency assumption. 
\end{remark}

Assuming that $P_\eta$ exists for each $\eta \subseteq \triangle$, we denote $X_\eta = \im (P_\eta)$ and
$$X^\eta = 
\begin{cases}
X_\emptyset & \eta = \emptyset \\
X_\eta \cap \bigcap_{\tau \subsetneqq \eta} \Ker (P_\tau) & \eta \neq \emptyset
\end{cases}.$$
The next theorem states that under suitable bounds on the angles between the $P_\sigma$'s for $\sigma \in \triangle (n-1)$ and the norms of the $P_\sigma$'s for $\sigma \in \triangle (n-1) \cup \triangle(n)$ , we have that 
$$X_\eta = \bigoplus_{\tau \subseteq \eta} X^\tau.$$
\begin{theorem}
\label{space decomposition}
Let $X$, $\triangle$, $P_\triangle$ and $P_\sigma$ for $\sigma \in \triangle (n-1)$ defined as above. Assume the following:
\begin{enumerate}
\item The projections $P_\sigma$ for $\sigma \in \triangle (n-1)$ are consistent.
\item For every $\tau \in \triangle (n-2)$, the projection $P_\tau$ exists.
\item $ \max_{\sigma \in \triangle (n-1) \cup \triangle (n)} \Vert P_\sigma \Vert \leq \beta_0,$ where $\beta_0 >1$ is the constant of Corollary \ref{Quick uniform convergence corollary}. 
\end{enumerate}
Then there is $\gamma >0$ such that if 
$$ \max \lbrace  \cos(\angle (P_\sigma,P_{\sigma '})) : \sigma, \sigma ' \in \triangle (n-1) \rbrace \leq \gamma,$$
then for every $\eta \subseteq \triangle$, $P_\eta$ exists and
$$X_\eta = \bigoplus_{\tau \subseteq \eta} X^\tau.$$
\end{theorem}
The proof of this theorem is based on a theorem similar to our Theorem \ref{small angle theorem} that appears in \cite{DymaraJ}[section 11] and the proof given there applies almost verbatim is our setting. We will repeat the proof below for completeness, but we claim no originality here.

\begin{lemma}
\label{In X^eta lemma}
Let $X$, $\triangle$, $P_\triangle$ and $P_\sigma$ for $\sigma \in \triangle (n-1)$ defined as above. Assume that the projections $P_\sigma$ for $\sigma \in \triangle (n-1)$ are consistent and that for every $\tau \subseteq \triangle$, $P_\tau$ exists. 

Fix $0 \leq i \leq n+1$ and assume that for every $\tau \subseteq \triangle$ with $\vert \tau \vert <i$ there is a projection $R_\tau: X \rightarrow X$ on $X^\tau$ such that $R_\tau = R_\tau P_\tau$. Then for every $\eta \subseteq \triangle$ with $\vert \eta \vert =i$ the following holds for every $v \in X_\eta$:
$$v \in X^\eta \Leftrightarrow \forall \tau \subsetneqq \eta, R_\tau v = 0.$$  
\end{lemma}

\begin{proof}
Assume first that $v \in X^\eta$, then by definition for every $\tau \subsetneqq \eta$, $v \in \Ker (P_\tau)$. By assumptions of the lemma $R_\tau P_\tau = R_\tau$ and therefore $R_\tau v =  R_\tau P_\tau v = 0$. 

In the other direction we will use induction on $\vert \eta \vert$. For $\vert \eta \vert =0$, $X^\emptyset=X_\emptyset$ and therefore the assertion of the lemma holds. Fix $0 <i < n+1$ and assume the lemma is true for every $\tau \subseteq \triangle$ with $\vert \tau \vert <i$. Fix $\eta \subseteq \triangle$ with $\vert \eta \vert =i$ and fix $v \in X_\eta$ such that for every $\tau \subsetneqq \eta$, $R_\tau v = 0$.  Let $\tau \subsetneqq \eta$ arbitrary. By the assumptions of the lemma for every $\tau ' \subsetneqq \tau$ the following holds: 
$$R_{\tau '} P_\tau v = R_{\tau '} P_{\tau '} P_\tau v.$$ 
By the consistency assumption (and remark \ref{P-triangle always commute remark}), $ P_{\tau '} P_{\tau} = P_{\tau'}$ and therefore
$$R_{\tau'} P_{\tau} v  = R_{\tau '} P_{\tau '} v = R_{\tau '} v =0.$$
By the induction assumption, we conclude that $P_\tau v \in X^\tau$. We also assumed that $R_\tau P_\tau v = R_\tau v=0$, therefore this yields that $P_\tau v =0$. We showed that for every $\tau \subsetneqq \eta$, $v \in \Ker (P_\tau)$ which implies that $v \in X^\eta$. 
\end{proof}

We will use the above lemma to prove Theorem \ref{space decomposition}:

\begin{proof}
Let $\varepsilon >0$ to be determined later and let $\gamma>0$ be the constant corresponding to $\varepsilon >0$ given by Theorem \ref{small angle theorem}. 

We shall prove that if $\varepsilon>0$ is small enough, then for each $0 \leq i \leq (n+1)$, there is a constant $C_i$ such that the following holds:
\begin{enumerate}
\item For each $\eta \subseteq \triangle$ with $\vert \eta \vert \leq i$, there is a projection $R_\eta : X \rightarrow X$ on $X^\eta$ such that $R_\eta P_\eta = R_\eta$ and $\Vert R_\eta \Vert \leq C_i$.
\item For every $0 \leq j \leq i$, $C_i \geq C_j$.
\item For every $\eta, \eta ' \subseteq \triangle$ such that $\eta \neq \eta'$ and $\vert \eta \vert, \vert \eta' \vert \leq i$, we have that $\Vert R_\eta R_{\eta'} \Vert \leq (C_{i})^2 \varepsilon$. 
\item For each $\eta \subseteq \triangle$ with $\vert \eta \vert =i$, $X_\eta =  \bigoplus_{\tau \subseteq \eta} X^\tau.$
\end{enumerate} 
The cases $i=0, i=1$ are straightforward:

For $i = 0$, we have that if $\vert \eta \vert = 0$, then $\eta = \emptyset$. Take $R_\emptyset = P_\emptyset$ and $C_0 = 4(n+1)+2$. We will check that for this choice conditions 1.-4. hold:
\begin{enumerate}
\item Note that 
$$R_\emptyset P_\emptyset =P_\emptyset P_\emptyset = P_\emptyset = R_\emptyset.$$
Also by Theorem \ref{small angle theorem}, $\Vert R_\emptyset \Vert \leq C_0$.
\item Holds vacuously.
\item Holds vacuously.
\item $X_\emptyset = X^\emptyset$. 
\end{enumerate}
 Also, by definition $X_\emptyset = X^\emptyset$.

For $i=1$, for $\eta \subseteq \triangle$ with $\vert \eta \vert =1$, take $R_\eta = P_\eta - P_\emptyset$ and $C_1 = 2 (4(n+1)+2)$. We will check that for this choice conditions 1.-4. hold:
\begin{enumerate}
\item Note that 
$$R_\eta = P_\eta - P_\emptyset = (I-P_\emptyset) P_\eta = (I-P_\emptyset) P_\eta  P_\eta = R_\eta P_\eta.$$
Also, by Theorem \ref{small angle theorem}, 
$$\Vert R_\eta \Vert \leq \Vert P_\eta \Vert + \Vert P_\emptyset \Vert \leq C_1.$$
\item $C_1 = 2 C_0 \geq C_0$.
\item Let $\eta, \eta' \subseteq \triangle$ such that $\vert \eta \vert, \vert \eta ' \vert \leq 1$ and $\eta \neq \eta '$. If $\eta = \emptyset$ or $\eta ' = \emptyset$, then $R_\eta R_{\eta'} =0$. If $\vert \eta \vert = \vert \eta ' \vert =1$, then $\eta \cap \eta' =\emptyset$ and
$$\Vert  R_\eta R_{\eta'} \Vert = \Vert (P_\eta - P_\emptyset) (P_{\eta '}  - P_\emptyset ) \Vert = \Vert P_\eta P_{\eta '} -P_{\eta \cap \eta ' } \Vert \leq \varepsilon \leq C_1^2 \varepsilon,$$
as needed. 
\item For every $\eta \subseteq \triangle$, such that $\vert \eta\vert =1$, $P_\eta - P_\emptyset$ is a projection on $X^\eta$ and therefore $X_\eta = X^\eta \oplus X^\emptyset$.
\end{enumerate} 

We proceed by induction. Let $i >1$ and assume that $(1),(2),(3),(4)$ above hold for every $j<i$. 

\textbf{Step 1 (proof of conditions 1.,2.):} Let $\eta \subseteq \triangle$ with $\vert \eta \vert = i$. We will show that $X_\eta$ is a sum of $X^\tau$ with $\tau \subseteq \eta$ and in doing so, we will find a projection operator $R_\eta : X \rightarrow X$ such that $\im (R_\eta) = X^\eta$ and $R_\eta P_\eta = R_\eta$. 

Let $d = 2^i-2$ and consider the $(d+1)$-valent tree such that each edge is labelled by some $\tau \subsetneqq \eta$ and no two edges with the same label meet at a vertex. Fix a vertex $x_0$ to be the root of this tree. Then for every vertex $x_j$ with distance $j >0$ from $x_0$ there is a path labelled $\tau_1,...,\tau_j$ from $x_0$ to $x_j$. For such $x_j$, define and operator $R(x_j) = (-1)^j R_{\tau_j} ... R_{\tau_1}$ and define $R(x_0) = I$. 
Denote the vertices of the tree by $V$ and define 
$$L_\eta = \sum_{x \in V} R(x).$$
Let $x_j$ be a vertex with distance $j>0$ from $x_0$. By the induction assumption $(3)$ we have that $\Vert R(x_j) \Vert \leq (C_{i-1}^2 \varepsilon )^{j-1} C_{i-1}$. Therefore if we choose $\varepsilon \leq \frac{1}{2 d C_{i-1}^{2}}$, then for every $v \in X_\eta$, $ \sum_{x} R(x) v$ is absolutely convergent:
$$\sum_{x} \Vert R(x) v \Vert = (1+ (d+1) C_{i-1} \sum_{j=1}^ \infty (d C_{i-1}^2 \varepsilon )^{j-1} ) \Vert v \Vert \leq (1 + 2(d+1) C_{i-1}) \Vert v \Vert.$$
Therefore $L_\eta$ is well defined if $\varepsilon$ is sufficiently small. For every $\tau \subsetneqq \eta$, denote
$$B_\tau = \lbrace x \in V \setminus \lbrace x_0 \rbrace \text{ such that the path from } x_0 \text{ to } x \text{ begins with } \tau \rbrace, $$
$$E_\tau = \lbrace x \in V \setminus \lbrace x_0 \rbrace \text{ such that the path from } x_0 \text{ to } x \text{ ends with } \tau \rbrace. $$
Then for a every $\tau \subsetneqq \eta$, we have that 
\begin{dmath*}
L_\eta = \sum_{x \in E_\tau} R(x) +\sum_{x \in V \setminus E_\tau} R(x) = -R_\tau (\sum_{x \in V \setminus E_\tau} R(x)) + \sum_{x \in V \setminus E_\tau} R(x).  
\end{dmath*}
Therefore, for every $\tau \subsetneqq \eta$, $R_\tau L_\eta = 0$ and therefore by Lemma \ref{In X^eta lemma} above, for every $v \in X_\eta$, $L_\eta v \in X^\eta$. This shows that $\im (L_\eta) \subseteq X^\eta$. To see that $\im (L_\eta) = X^\eta$, notice that for every $v \in X^\eta$ and for every $\tau \subsetneqq \eta$, $R_\tau v=0$ and therefore by the definition of $L_\eta$, $L_\eta v =v$. 

We will take $L_\eta P_\eta$ as our candidate for $R_\eta$ and take $C_i = (4(n+1)+2) (1 + 2(d+1) C_{i-1})$ as a bound on $\Vert R_\eta \Vert$ (we showed above that $\Vert L_\eta \Vert \leq 1 + 2(d+1) C_{i-1}$). Notice that $C_i$ was chosen such that $C_i \geq C_{i-1}$ as needed. It is clear that taking $R_\eta = L_\eta P_\eta$ implies that $R_\eta P_\eta = R_\eta$.

To show that $R_\eta$ is indeed a projection, notice first that for every $\tau \subsetneqq \eta$, we have that
\begin{dmath*}
L_\eta = \sum_{x \in B_\tau} R(x) +\sum_{x \in V \setminus B_\tau} R(x) = - (\sum_{x \in V \setminus B_\tau} R(x)) R_\tau + \sum_{x \in V \setminus B_\tau} R(x).  
\end{dmath*}
Therefore, for every $\tau \subsetneqq \eta$, $L_\eta R_\tau = 0$. Second, notice that
\begin{dmath*}
L_\eta = I+ \sum_{\tau \subsetneqq \eta} \sum_{x \in E_\tau} R(x) = I- \sum_{\tau \subsetneqq \eta} R_\tau \sum_{x \in V \setminus E_\tau} R(x).
\end{dmath*}
Therefore 
$$L_\eta^2 = L_\eta (I- \sum_{\tau \subsetneqq \eta} R_\tau \sum_{x \in V \setminus E_\tau} R(x)) = L_\eta -\sum_{\tau \subsetneqq \eta} L_\eta R_\tau \sum_{x \in V \setminus E_\tau} R(x) = L_\eta.$$
This yields that $R_\eta^2 = R_\eta$. 

The same computation also shows that $X_\eta$ is a linear sum of $X^\tau$ with $\tau \subseteq \eta$. First, for every $v \in X_\eta$ we showed that $L_\eta v \in X^\eta$. Second, if we denote for every $\tau \subsetneqq \eta$, 
$$v^\tau = R_\tau \sum_{x \in V \setminus E_\tau} R(x) v,$$ 
then $v^\tau \in X^\tau$. Last, we showed above that 
$$L_\eta = I- \sum_{\tau \subsetneqq \eta} R_\tau \sum_{x \in V \setminus E_\tau} R(x),$$
and this yields that for every $v \in X_\eta$, 
$$v = \sum_{\tau \subsetneqq \eta} v^\tau + L_\eta v,$$
as needed.

\textbf{Step 2 (proof of condition 3.):} We will show that for every $\eta, \eta' \subseteq \triangle$, with $\vert \eta \vert, \vert \eta ' \vert \leq i$ and $\eta \neq \eta'$, we have that $\Vert R_\eta R_{\eta '} \Vert \leq (C_{i})^2 \varepsilon$. We'll split the proof of this fact into several cases.

In the case that $\eta \cap \eta ' \subsetneqq \eta '$, notice that $\im (P_{\eta '}) \cap \Ker (P_{\eta \cap \eta '}) \subseteq \im (R_{\eta '})$ and therefore
$$R_\eta  R_{\eta '} = R_\eta P_\eta P_{\eta'} (I- P_{\eta \cap \eta'}) R_{\eta '}.$$
This yields that 
$$\Vert R_\eta  R_{\eta '} \Vert \leq \Vert R_\eta \Vert \Vert R_{\eta '} \Vert \cos (\angle (P_\eta, P_{\eta ' } )) \leq C_{\vert \eta \vert} C_{\vert \eta ' \vert} \varepsilon  \leq (C_i)^2 \varepsilon,$$
as needed.

In the case that $\eta \subseteq \eta'$, we have that $\im (R_{\eta '}) \subseteq \Ker (P_\eta)$ and therefore 
$$R_\eta R_{\eta '} = R_\eta P_\eta R_{\eta '} = 0.$$

In the case that $\eta ' \subseteq \eta$ and $\vert \eta \vert \leq i-1, \vert \eta ' \vert \leq i-1$, the inequality follows from the induction assumption.

We are left with the case in which $\vert \eta \vert =i$ and $\eta ' \subseteq \eta$. In this case, by step 1 above, $L_\eta R_{\eta '} =0$ and therefore $R_\eta R_{\eta '} =0$ and we are done.

\textbf{Step 3 (proof of condition 4.):} We will finish by showing that given that $\varepsilon >0$ is small enough, 
$$X_\eta = \bigoplus_{\tau \subseteq \eta} X^\tau.$$
We already showed in step 1, that $X_\eta$ is a linear sum of $X^\tau$ such that $\tau \subseteq \eta$. Assume there are $v^\tau \in X^\tau$ such that 
$$\sum_{\tau \subseteq \eta} v^\tau =0.$$
Let $\tau'$ be such that for every $\tau \subseteq \sigma$, $\Vert v^{\tau'} \Vert \geq \Vert v^{\tau} \Vert$. Then $R_{\tau'} (\sum_{\tau \subseteq \eta} v^\tau) =0$. Using the bound on the norm of $\Vert R_{\tau'} R_{\tau} \Vert$ established in step 2, this yields
\begin{dmath*}
0= \Vert R_{\tau'} (\sum_{\tau \subseteq \eta} v^\tau) \Vert \geq \Vert v^{\tau'} \Vert - \Vert \sum_{\tau \subseteq \eta, \tau \neq \tau '} R_{\tau'} v^\tau  \Vert \geq  \\
\Vert v^{\tau'} \Vert - \sum_{\tau \subseteq \eta, \tau \neq \tau'} \Vert R_{\tau'} R_{\tau} v^\tau \Vert \geq \\
\Vert v^{\tau'} \Vert - (C_i)^2 \varepsilon \Vert v^{\tau'} \Vert = \Vert v^{\tau'} \Vert (1-(2^{i}-1) (C_i)^2 \varepsilon).
\end{dmath*}
Therefore, if $\varepsilon$ is chosen such that $\varepsilon < \frac{1}{(2^{i}-1) (C_i)^2}$, we get that $\Vert v^{\tau'} \Vert =0$ and therefore $v^\tau =0$ for every $\tau \subseteq \eta$. This yields  
$$X_\eta = \bigoplus_{\tau \subseteq \eta} X^\tau,$$
as needed.
\end{proof} 

\begin{remark}
Note that in the above proof, the consistency assumption is crucial in  the proof of Lemma \ref{In X^eta lemma} which in turn was crucial for step 1 of the above proof.
\end{remark}

\section{Vanishing of cohomology}

Let $\Sigma$ be a pure $n$-dimensional infinite simplicial complex and let $G < Aut (\Sigma)$ be a closed subgroup. Assume that $(\Sigma, G)$ satisfies conditions $(\mathcal{B} 1)-(\mathcal{B} 4)$ defined in subsection \ref{Groups acting on simplicial complexes subscetion} above. Assume further that all the $1$-dimensional links of $\Sigma$ are compact. Fix a chamber $\triangle \in \Sigma (n)$. Let $\mu$ be the Haar measure on $G$. For $-1 \leq i \leq n$, denote $\triangle (i)$ to be the $i$-dimensional simplices of $\triangle$. 
For $\sigma \in (\triangle (n) \cup \triangle (n-1) \cup \triangle (n-2))$ define $k_\sigma \in C_c (G)$ as
$$k_\sigma = \frac{\chi_{G_\sigma}}{\mu (G_\sigma)},$$
where $\chi_{G_\sigma}$ is the indicator function on $G_\sigma$ (note that by our assumptions $G_\sigma$ is a compact group). Observe that 
\begin{itemize}
\item For $\sigma, \tau \in (\triangle (n) \cup \triangle (n-1) \cup \triangle (n-2))$, if $\tau \subset \sigma$, then $k_\tau k_\sigma = k_\tau$.
\item For any continuous representation $\pi$ of $G$ on a Banach space $X$ and any $\sigma \in (\triangle (n) \cup \triangle (n-1) \cup \triangle (n-2))$, $\pi (k_\sigma)$ is a projection on the $X^{\pi (G_\sigma)}$ (recall that $X^{\pi (G_\sigma)}$ is the subspace of vectors fixed by $G_\sigma$).
\end{itemize}

These observations yields that for any two $\sigma, \sigma ' \in \triangle (n-1)$ and any representation $\pi$ of $G$, we can define the cosine of the angle between $\pi (k_\sigma)$ and $\pi (k_{\sigma'})$ as in definition \ref{angle between projections definition} above:
\begin{dmath*}
{\cos (\angle (\pi (k_\sigma), \pi (k_{\sigma'}))) =} \\
{ \max \lbrace \Vert \pi (k_\sigma) \pi ( k_{\sigma '}) - \pi (k_{\sigma \cap \sigma '} ) \Vert, \Vert \pi (k_{\sigma'}) \pi ( k_{\sigma }) - \pi (k_{\sigma \cap \sigma '} ) \Vert \rbrace =} \\
\max \lbrace \Vert \pi (k_\sigma k_{\sigma '} - k_{\sigma \cap \sigma '} ) \Vert, \Vert \pi (k_{\sigma'} k_{\sigma } - k_{\sigma \cap \sigma '} ) \Vert \rbrace.
\end{dmath*}

Therefore we are in the setting of Theorem \ref{space decomposition}. Applying Theorem \ref{space decomposition} combined with Theorem \ref{general vanishing of cohomology based on decomposition} yields the following:
\begin{theorem}
\label{vanishing of cohomology by conditions on the representations}
Let $\Sigma$ be a pure $n$-dimensional infinite simplicial complex and let $G < Aut (\Sigma)$ be a closed subgroup. Assume that $(\Sigma, G)$ satisfy conditions $(\mathcal{B} 1)-(\mathcal{B} 4)$ and that there is $l \in \mathbb{N}$ such that all the $l$-dimensional links of $\Sigma$ are compact. Then there are constants $\gamma = \gamma (n) >0$, $\beta = \beta (n) >1$ such that for every representation $\pi$ of $G$ on a Banach space, if
$$\sup_{\sigma \in \triangle (n-1)} \Vert \pi (k_\sigma) \Vert \leq \beta , \sup_{\sigma, \sigma ' \in \triangle (n-1)} \cos (\angle (\pi (k_\sigma), \pi (k_{\sigma '}))) \leq \gamma,$$
and the projections $\pi (k_\sigma)$ with $\sigma \in \triangle (n-1)$ are consistent, then 
$$H^* (G, \pi) = \bigoplus_{\eta \subseteq \triangle} \widetilde{H}^{*-1} (D_\eta; X^{\eta}),$$
and 
$$H^i (G,\pi) = 0 \text{ for } i=1,...,l.$$
\end{theorem}

\begin{proof}
Denote $P_\sigma = \pi (k_\sigma)$ for $\sigma \in \triangle (n-2) \cup \triangle (n-1) \cup \triangle (n)$. Let $\beta = \beta_0>1$ and $\gamma$ as in Theorem \ref{space decomposition}. The assumptions on $\pi$ grantee that the $P_\sigma$'s fulfil the conditions of Theorem \ref{space decomposition}.

Therefore for every $\eta \subseteq \triangle$, $X_\eta = \bigoplus_{\tau \subseteq \eta} X^\tau$. The vanishing of cohomology follows from Theorem \ref{general vanishing of cohomology based on decomposition}.

Note that the constants $\gamma$, $\beta$ depend only on the dimension $n$ (and not on any other characteristics of $\Sigma$).
\end{proof}

We will show that there are sufficient conditions that grantee the fulfilment of the conditions of the theorem above in a class of representations $\mathcal{F}_0 (\overline{\mathcal{E}_3 (\mathcal{E}_2 (\mathcal{E}_1 (r, C_1),\theta_2),C_3)}, G, s_0)$ defined in the introduction for suitable choices of $s_0>0$ and $r$. We start by recalling the following result from \cite{ORobust} that connects the Schatten norm of the projection operators to condition $(\mathcal{B}_{\delta, r})$ defined above:

\begin{lemma}\cite{ORobust}[Corollary 4.20]
\label{condition r, delta to schatten norm lemma}
Let $\Sigma$ be a pure $n$-dimensional infinite simplicial complex and let $G < Aut (\Sigma)$ be a closed subgroup. Assume that $(\Sigma, G)$ satisfy conditions $(\mathcal{B} 1)-(\mathcal{B} 4)$ and condition $(\mathcal{B}_{\delta, r})$, then for every $\sigma, \sigma' \in \triangle (n-1)$, 
$$\Vert \lambda (k_\sigma k_{\sigma' } - k_{\sigma \cap \sigma'}) \Vert_{S^r} \leq \delta,$$
where  $\lambda \in B(L^2 (G_{\sigma \cap \sigma'}, \mu))$ is the left regular representation.
\end{lemma}

Using the above lemma, we are able to deduce arbitrary small angles between all the projections $\pi (k_\sigma)$ and $\pi (\sigma')$ given the condition $(\mathcal{B}_{\delta,r'})$ is fulfilled:

\begin{lemma}
\label{small angle by B-delta,r condition}
Let $\Sigma$ be a pure $n$-dimensional infinite simplicial complex and let $G < Aut (\Sigma)$ be a closed subgroup. Assume that $(\Sigma, G)$ satisfy conditions $(\mathcal{B} 1)-(\mathcal{B} 4)$ and that the $1$-dimensional links of $\Sigma$ are finite.

Let $r > 2, C_1 \geq 1, 1 \geq \theta_2 > 0, C_{3} \geq 1$ be constants. For every $\gamma >0$, $s_0 \geq 0$, $2 \leq r' < r$, there is a $\delta >0$ such that if $(\Sigma, G)$ satisfies condition $(\mathcal{B}_{\delta,r'})$, then for every $\pi \in \mathcal{F} (\overline{\mathcal{E}_3 (\mathcal{E}_2 (\mathcal{E}_1 (r, C_1),\theta_2),C_3)}, G, s_0)$,
$$\sup_{\sigma, \sigma ' \in \triangle (n-1)} \cos (\angle (\pi (k_\sigma), \pi (k_{\sigma '}))) \leq \gamma.$$
\end{lemma}

\begin{proof}
Fix $\pi \in \mathcal{F} (\overline{\mathcal{E}_3 (\mathcal{E}_2 (\mathcal{E}_1 (r, C_1),\theta_2),C_3)}, G, s_0)$. Let $\sigma, \sigma' \in \triangle (n-1)$ be any two different $(n-1)$-dimensional faces of $\triangle$ and assume without loss of generality that
$$\cos (\angle (\pi (k_\sigma), \pi (k_{\sigma '}))) =  \Vert \pi (k_\sigma k_{\sigma '} - k_{\sigma \cap \sigma '} ) \Vert.$$
By Proposition \ref{bounding the norm of pi(f) - proposition}, we have that
$$\Vert \pi (k_\sigma k_{\sigma '} - k_{\sigma \cap \sigma '} ) \Vert \leq e^{2s_0} \Vert \lambda (k_\sigma k_{\sigma '} - k_{\sigma \cap \sigma '} ) \otimes id_X \Vert_{B(L^2 (G_{\sigma \cap \sigma'} ;X))}.$$
Note that for any Banach space $X$, we have that 
$$\Vert \lambda (k_\sigma k_{\sigma '} - k_{\sigma \cap \sigma '} ) \otimes id_X \Vert_{B(L^2 (G_{\sigma \cap \sigma'} ;X))} \leq \Vert \lambda (k_\sigma k_{\sigma '} - k_{\sigma \cap \sigma '} ) \Vert_{B(L^1 (G_{\sigma \cap \sigma'}))} \leq 2.$$
Assuming that $\delta \leq 1$ and applying Lemma \ref{condition r, delta to schatten norm lemma} and Corollary \ref{bounding norm of composition of deformations} (with $L=2$) yields that 
$$ \Vert \lambda (k_\sigma k_{\sigma '} - k_{\sigma \cap \sigma '} ) \otimes id_X \Vert_{B(L^2 (G_{\sigma \cap \sigma'} ;X))} \leq C_3 2 (C \delta)^{\theta_2},$$
where $C=C(C_1,r,r')$ is the constant given in Corollary \ref{bounding norm of composition of deformations}. Therefore, we have that for every $\pi \in \mathcal{F} (\overline{\mathcal{E}_3 (\mathcal{E}_2 (\mathcal{E}_1 (r, C_1),\theta_2),C_3)}, G, s_0)$, 
$$\cos (\angle (\pi (k_\sigma), \pi (k_{\sigma '}))) \leq e^{2s_0} C_3 2 (C \delta)^{\theta_2},$$
and choosing $\delta = \frac{1}{C} (\frac{\gamma}{2e^{2s_0} C_3} )^{\frac{1}{\theta_2}}$ yields the needed inequality.
\end{proof}

The implication of the above lemma is that when applying Theorem \ref{vanishing of cohomology by conditions on the representations} on a class of representations of the form $\mathcal{F} (\overline{\mathcal{E}_3 (\mathcal{E}_2 (\mathcal{E}_1 (r, C_1),\theta_2),C_3)}, G, s_0)$, one can replace the condition 
$$\sup_{\sigma, \sigma ' \in \triangle (n-1)} \cos (\angle (\pi (k_\sigma), \pi (k_{\sigma '}))) \leq \gamma,$$
by the condition $(\mathcal{B}_{\delta,r'})$ for suitable values of $\delta$ and $r'$:

\begin{theorem}
\label{vanishing of cohomology by conditions on the links and consistency assumption}
Let $\Sigma$ be a pure $n$-dimensional infinite simplicial complex and let $G < Aut (\Sigma)$ be a closed subgroup. Assume that $(\Sigma, G)$ satisfy conditions $(\mathcal{B} 1)-(\mathcal{B} 4)$ and that there is a $l \geq 1$ such that all the $l$-dimensional links of $\Sigma$ are compact. 

Let $r > r' \geq 2, C_1 \geq 1, 1 \geq \theta_2 > 0, C_{3} \geq 1$ be constants. Then there are $s_0 = s_0 (n)>0$ and  
$$\delta = \delta (n,r,r',C_{1},\theta_2, C_{3})  >0$$ 
such that if $(\Sigma, G)$ fulfil condition $(\mathcal{B}_{\delta,r'})$ and if the projections $\pi (k_\sigma)$ with $\sigma \in \triangle (n-1)$ are consistent, then 
$$H^* (G, \pi) = \bigoplus_{\eta \subseteq \triangle} \widetilde{H}^{*-1} (D_\eta; X^{\eta}),$$
and
$$H^i (G,\pi) = 0 \text{ for } i=1,...,l,$$
for every $\pi \in \mathcal{F} (\overline{\mathcal{E}_3 (\mathcal{E}_2 (\mathcal{E}_1 (r, C_1),\theta_2),C_3)}, G, s_0)$.
\end{theorem} 

\begin{proof}
Let $\beta >1$, $\gamma >0$ be the constants given by Theorem \ref{vanishing of cohomology by conditions on the representations}.

Choose $s_0 = \ln (\beta)$, by this choice the inequality
$$\max_{\sigma \in \triangle (n-1)} \Vert \pi (k_\sigma) \Vert \leq \sup_{g \in \bigcup_{\sigma \in \triangle (n-1)} G_\sigma} \Vert \pi (g) \Vert  \leq  \sup_{g \in \bigcup_{\tau \in \triangle (n-2)} G_\tau} \Vert \pi (g) \Vert \leq e^{s_0} = \beta$$
is satisfied for each $\pi \in \mathcal{F} (\overline{\mathcal{E}_3 (\mathcal{E}_2 (\mathcal{E}_1 (r, C_1),\theta_2),C_3)}, G, s_0)$. 

By Lemma \ref{small angle by B-delta,r condition}, we can choose $\delta>0$ small enough such that the condition $(\mathcal{B}_{\delta,r'})$ we imply that 
$$\sup_{\sigma, \sigma ' \in \triangle (n-1)} \cos (\angle (\pi (k_\sigma), \pi (k_{\sigma '}))) \leq \gamma,$$
for every $\pi \in \mathcal{F} (\overline{\mathcal{E}_3 (\mathcal{E}_2 (\mathcal{E}_1 (r, C_1),\theta_2),C_3)}, G, s_0)$. 

Therefore for this choice of $s_0>0$ and $\delta >0$, the conditions of Theorem \ref{vanishing of cohomology by conditions on the representations} are fulfilled and the conclusion follows.
\end{proof}

The unsatisfactory part of the above theorem is the assumption of consistency of the projections $\pi (k_\sigma)$. We will show that when passing to the class $\mathcal{F}_0 (\overline{\mathcal{E}_3 (\mathcal{E}_2 (\mathcal{E}_1 (r, C_1),\theta_2),C_3)}, G, s_0)$ (in which the dual representations are continuous) this always assumption holds.  

\begin{lemma}
\label{continuous dual implies consistency}
Let $\Sigma$ be a pure $n$-dimensional infinite simplicial complex and let $G < Aut (\Sigma)$ be a closed subgroup. Assume that $(\Sigma, G)$ satisfy conditions $(\mathcal{B} 1)-(\mathcal{B} 4)$.

Let $\pi$ be a continuous representation on a Banach space $X$ such that 
$$\sup_{\sigma \in \triangle (n-1)} \Vert \pi (k_\sigma) \Vert \leq \beta_0 , \sup_{\sigma, \sigma ' \in \triangle (n-1)} \cos (\angle (\pi (k_\sigma), \pi (k_{\sigma '}))) \leq \gamma_0,$$
where $\beta_0 >1, \gamma_0 >0$ are the constants given by Corollary \ref{Quick uniform convergence corollary}.

If the dual representation $\pi^*$ is continuous, then the projections $\pi (k_\sigma)$ for $\sigma \in \triangle (n-1)$ are consistent. 
\end{lemma}

\begin{proof}
Let $\mathcal{F} = \lbrace \pi, \pi^* \rbrace$. Note that for $\pi^*$ the following holds:
$$\sup_{\sigma \in \triangle (n-1)} \Vert \pi^* (k_\sigma) \Vert \leq \beta_0 , \sup_{\sigma, \sigma ' \in \triangle (n-1)} \cos (\angle (\pi^* (k_\sigma), \pi^* (k_{\sigma '}))) \leq \gamma_0.$$ 
Therefore by Corollary \ref{Quick uniform convergence corollary}, for every $\tau \subsetneqq \triangle$,
$$\left(\frac{1}{n+1-\vert \tau \vert} \sum_{\sigma \in \triangle (n-1), \tau \subseteq \sigma} k_\sigma \right)^i \xrightarrow{i \rightarrow \infty} k_\tau ,$$
where the convergence is in $C_{\mathcal{F}}$ and $\pi (k_\tau)$ and $\pi^* (k_\tau)$ are projections on $X^{\pi(G_\tau)}$ and $(X^*)^{\pi^* (G_\tau)}$ respectively. 

By Proposition \ref{consistency is checked on P_i's proposition}, in order to prove consistency, it is enough to show that 
$$\forall \tau \subsetneqq \triangle, \forall \sigma \in \triangle (n-1), \tau \subseteq \sigma \Rightarrow \pi (k_\tau) \pi (k_\sigma) = \pi (k_\tau).$$ 

We will prove the following condition which is actually stronger: 
$$\forall \tau \subsetneqq \triangle, \forall g \in G_\tau, \pi (k_\tau) \pi (g) = \pi (k_{\tau}).$$
Fix some $\tau \subsetneqq \triangle$ and $g \in G_\tau$. For every $v \in X, w \in X^*$ we have that 
\begin{dmath*}
\langle \pi (k_\tau) \pi (g).v, w \rangle = \langle v,  \pi^* (g) \pi^* (k_\tau).w \rangle = \langle v, \pi^* (k_\tau).w \rangle = \langle \pi (k_\tau).v, w \rangle.
\end{dmath*}
Therefore, $\pi (k_\tau) \pi (g) = \pi (k_{\tau})$ as needed.
\end{proof}

\begin{remark}
We note that if $G$ is a discrete group, then the condition of $\pi^*$ being continuous always holds (since it is vacuous). 
\end{remark}

As a corollary of the Lemma \ref{continuous dual implies consistency} we deduce the following theorem:

\begin{theorem}
\label{vanishing of cohomology by conditions on the links and cont dual assumption}
Let $\Sigma$ be a pure $n$-dimensional infinite simplicial complex and let $G < Aut (\Sigma)$ be a closed subgroup. Assume that $(\Sigma, G)$ satisfy conditions $(\mathcal{B} 1)-(\mathcal{B} 4)$ and that there is $l \in \mathbb{N}$ such that all the $l$-dimensional links of $\Sigma$ are compact. 

Let $r > r' \geq 2, C_1 \geq 1, 1 \geq \theta_2 > 0, C_{3} \geq 1$ be constants. Then there are $s_0 = s_0 (n)>0$ and  
$$\delta = \delta (n,r,r',C_{1},\theta_2, C_{3})  >0$$ 
such that if $(\Sigma, G)$ fulfil condition $(\mathcal{B}_{\delta,r'})$, then 
$$H^* (G, \pi) = \bigoplus_{\eta \subseteq \triangle} \widetilde{H}^{*-1} (D_\eta; X^{\eta}),$$
and
$$H^i (G,\pi) = 0 \text{ for } i=1,...,l,$$
for every $\pi \in \mathcal{F}_0 (\overline{\mathcal{E}_3 (\mathcal{E}_2 (\mathcal{E}_1 (r, C_1),\theta_2),C_3)}, G, s_0)$.
\end{theorem} 

\begin{proof}
By Lemma \ref{continuous dual implies consistency}, the projections $\pi (k_\sigma)$ with $\sigma \in \triangle (n-1)$ are consistent for every $\pi \in \mathcal{F}_0 (\overline{\mathcal{E}_3 (\mathcal{E}_2 (\mathcal{E}_1 (r, C_1),\theta_2),C_3)}, G, s_0)$ and therefore we can apply Theorem \ref{vanishing of cohomology by conditions on the links and consistency assumption}.
\end{proof}

We recall that by Theorem \ref{Asplund implies continuous dual rep} stated above, for a continuous representation $\pi$ on a Banach space $X$, if $X$ is an Asplund space then $\pi^*$ is continuous. We also recall that all reflexive Banach spaces are Asplund spaces and it was shown in the introduction that the subclass of reflexive Banach spaces of $\overline{\mathcal{E}_3 (\mathcal{E}_2 (\mathcal{E}_1 (r, C_1),\theta_2),C_3)}$ contains several interesting families of Banach spaces.

As stated in subsection \ref{Groups acting on simplicial complexes subscetion} above, the main example of couples $(\Sigma, G)$ satisfying the conditions  $(\mathcal{B} 1)-(\mathcal{B} 4)$ are groups $G$ with a BN-pair acting on a building $\Sigma$. In Proposition \ref{condition B-delta,r for buildings}, we showed that the condition $(\mathcal{B}_{\delta,r})$ can also be deduced for these examples for suitable values of $r$. Therefore we can deduce the following corollary:

\begin{corollary}
\label{vanishing of cohomology for BN pair groups}
Let $G$ be a group coming from a BN-pair and let $\Sigma$ be the $n$-dimensional building on which it acts. Assume that $n >1$ and there is some $l \geq 1$ such that all the $l$-dimensional links of $\Sigma$ are compact. Denote by $q$ the thickness of the building $\Sigma$ and let $m'$ be the smallest integer such that all the links of $1$-dimensional links of $\Sigma$ are generalized $m$-gons with $m \leq m'$. 

Let 
$$r > \begin{cases}
4 & m' =3 \\
8 & m' =4 \\
18 & m' = 6 \\
20 & m' =8
\end{cases}$$ 
and $C_{1} \geq 1, 1 \geq \theta_2 > 0, C_{3} \geq 1$ be constants, then there are $s_0 = s_0 (n)>0$ and  
$$Q = Q (n,r,m',C_{1},\theta_2, C_{3},s_0)  \in \mathbb{N}$$
such that if $q \geq Q$, then 
$$H^* (G, \pi) = \bigoplus_{\eta \subseteq \triangle} \widetilde{H}^{*-1} (D_\eta; X^{\eta}),$$
and
$$H^i (G,\pi) = 0 \text{ for } i=1,...,l,$$
for every $\pi \in \mathcal{F}_0 (\overline{\mathcal{E}_3 (\mathcal{E}_2 (\mathcal{E}_1 (r, C_1),\theta_2),C_3)}, G, s_0)$.
\end{corollary}

\begin{proof}
Fix 
$$r' = \begin{cases}
\frac{4+r}{2} & m' =3 \\
\frac{8+r}{2}  & m' =4 \\
\frac{18+r}{2}  & m' = 6 \\
\frac{20+r}{2}  & m' =8
\end{cases}$$ 
and let $s_0 (n) >0$, $\delta = \delta (n,r,r',C_{1},\theta_2, C_{3},s_0) >0$ be as in Theorem \ref{vanishing of cohomology by conditions on the links and cont dual assumption}. By Proposition \ref{condition B-delta,r for buildings}, there is a large enough $Q$ such that for every $q \geq Q$, $\Sigma$ fulfils the condition $(\mathcal{B}_{\delta,r'})$ and we are done by Theorem \ref{vanishing of cohomology by conditions on the links and cont dual assumption}.
\end{proof}

\appendix
\section{Angle between projections without the consistency assumption}
Under the notations of definition \ref{angle between projections definition}, given $\tau, \tau ' \subseteq \triangle$ such that $P_\tau, P_{\tau'}, P_{\tau \cap \tau'}$ exist, we can define $\cos (\angle (P_\tau, P_{\tau '}))$ as 
$$\cos (\angle (P_\tau, P_{\tau '})) = \max \lbrace \Vert P_\tau P_{\tau'} (I-P_{\tau \cap \tau'}) \Vert, \Vert P_{\tau'} P_{\tau} (I-P_{\tau \cap \tau'}) \Vert \rbrace.$$
We note that in this definition, we do not assume that $P_{\tau \cap \tau'} P_\tau = P_{\tau \cap \tau'}$ or that $P_{\tau \cap \tau'} P_{\tau '} = P_{\tau \cap \tau'}$. However, even without this assumption of consistency, we can derive a theorem similar to Theorem \ref{small angle theorem}:

\begin{theorem}
\label{small angle theorem without consistency assumption theorem}
Let $X$, $\triangle$ and $P_\sigma \in \triangle (n-1) \cup \triangle (n)$ be as in definition \ref{simplex projections definition} above. 
Assume that for every $\eta \in \triangle (n-2)$, the projection $P_\eta$ exists and for $\sigma \in \triangle (n-1)$, if $\eta \subset \sigma$ then $P_\eta P_\sigma = P_\eta$.

Then there is $\beta >1$ such that for every $\varepsilon >0$ there is $\gamma>0$ such that if 
$$ \max_{\sigma \in \triangle (n-1)} \lbrace \Vert P_\sigma \Vert \leq \beta \text{ and } \max \lbrace  \cos(\angle (P_\sigma,P_{\sigma ' })) : \sigma, \sigma ' \in \triangle (n-1) \rbrace \leq \gamma,$$
then $P_\tau$ exist for any $\tau \subseteq \triangle$ and for every $\tau, \tau' \subseteq \triangle$,
$$\cos (\angle (P_\tau,P_{\tau '})) \leq \varepsilon.$$

\end{theorem}

We will start with the following lemma asserting that under the assumptions of the above theorems the projections are bounded and ``almost'' commute: 
\begin{lemma}
\label{almost commutativity lemma}
Let $X$, $\triangle$ and $P_\sigma \in \triangle (n-1) \cup \triangle (n)$ be as in definition \ref{simplex projections definition} above. 
Assume that for every $\eta \in \triangle (n-2)$, the projection $P_\eta$ exists and for $\sigma \in \triangle (n-1)$, if $\eta \subset \sigma$ then $P_\eta P_\sigma = P_\eta$.

Then there is $\beta >1$ such that for every $\varepsilon >0$ there is $\gamma>0$ such that if 
$$ \max_{\sigma \in \triangle (n-1)} \lbrace \Vert P_\sigma \Vert \leq \beta \text{ and } \max \lbrace  \cos(\angle (P_\sigma,P_{\sigma ' })) : \sigma, \sigma ' \in \triangle (n-1) \rbrace \leq \gamma.$$
then the following holds:
\begin{enumerate}
\item For every $\tau \subseteq \triangle$, $P_\tau$ exists and $\Vert P_\tau \Vert \leq 4(n+1)+2$.
\item For every $\tau, \tau' \subseteq \triangle$, $\Vert P_\tau P_{\tau'} - P_{\tau'} P_\tau \Vert \leq \varepsilon$.
\end{enumerate}
\end{lemma}

\begin{proof}
Take $\beta = \min \lbrace \beta_0, \frac{2(n+1)+1}{2(n+1)} \rbrace$ and $\gamma \leq \gamma_0$. The proof of the first assertion is identical to the one given in the proof of Theorem \ref{small angle theorem} above (note that the consistency assumption was not used in this proof).

We are left with proving the second assertion.  Fix $\varepsilon >0$ and assume that $\gamma \leq \gamma_0$. Then for every $\tau, \tau' \subseteq \triangle$ and every $i \in \mathbb{N}$ the following holds:
\begin{dmath*}
\Vert P_\tau P_{\tau'} - T_\tau^i T_{\tau'}^i \Vert \leq \Vert P_\tau  (P_{\tau'} - T_{\tau'}^i) \Vert + \Vert (P_\tau - T_\tau^i) T_{\tau'}^i \Vert \leq \\
 \Vert P_\tau \Vert \Vert P_{\tau'} - T_{\tau'}^i \Vert + \Vert P_\tau - T_\tau^i \Vert \Vert T_{\tau'} \Vert^i. 
\end{dmath*}
Note that $\Vert T_{\tau'} \Vert \leq \beta \leq \frac{2(n+1)+1}{2(n+1)}$ and that $\Vert P_\tau \Vert \leq 4(n+1)+2$. Combining these bounds with Corollary \ref{Quick uniform convergence corollary} yields
\begin{dmath*}
\Vert P_\tau P_{\tau'} - T_\tau^i T_{\tau'}^i \Vert \leq (4(n+1)+2) 4(n+1) \left( \frac{2(n+1)-1}{2(n+1)} \right)^{i-1} + (4(n+1)+2) \left( \frac{4(n+1)^2-1}{4(n+1)^2} \right)^{i-1}.  
\end{dmath*} 
The right-hand side of the above inequality goes to $0$ as $i$ tends to $\infty$ and therefore we can choose $i_0$ such that $\Vert P_\tau P_{\tau'} - T_\tau^{i_0} T_{\tau'}^{i_0} \Vert \leq \frac{\varepsilon}{4}$ (note that this choice of $i_0$ holds for every $\gamma \leq \gamma_0$). Similarly, $\Vert P_{\tau '} P_{\tau} - T_{\tau '}^{i_0} T_{\tau}^{i_0} \Vert \leq \frac{\varepsilon}{4}$ and therefore 
$$\Vert P_\tau P_{\tau'} - P_{\tau '} P_{\tau} \Vert \leq \dfrac{\varepsilon}{2} + \Vert T_{\tau}^{i_0} T_{\tau'}^{i_0} - T_{\tau'}^{i_0} T_{\tau}^{i_0} \Vert.$$
We are left to prove that by choosing $\gamma$ small enough, we can assure that 
$$\Vert T_{\tau}^{i_0} T_{\tau'}^{i_0} - T_{\tau'}^{i_0} T_{\tau}^{i_0} \Vert \leq \dfrac{\varepsilon}{2},$$
when $i_0$ is fixed. As in the proof of Theorem \ref{angle between several projections theorem}, we note that for any $\sigma, \sigma ' \in \triangle (n-1)$, 
$$\Vert P_\sigma P_{\sigma'} - P_{\sigma'} P_\sigma \Vert \leq 2 \gamma.$$
Therefore $\Vert T_{\tau} T_{\tau'} - T_{\tau'} T_{\tau} \Vert \leq 2 \gamma$. By permuting pairwise $T_\tau$ and $T_{\tau'}$ we get that 
$$\Vert T_{\tau}^{i_0} T_{\tau'}^{i_0} - T_{\tau'}^{i_0} T_{\tau}^{i_0} \Vert \leq 2 i_0^2 \Vert T_\tau \Vert^{i_0-1} \Vert T_{\tau'} \Vert^{i_0-1} \gamma \leq 2 i_0^2 \left( \frac{2(n+1)+1}{2(n+1)} \right)^{i_0} \gamma.$$
Recall that $i_0$ is fixed and therefore we can choose $\gamma$ small enough such that 
$$\Vert T_{\tau}^{i_0} T_{\tau'}^{i_0} - T_{\tau'}^{i_0} T_{\tau}^{i_0} \Vert \leq \dfrac{\varepsilon}{2},$$
and we are done. 
\end{proof}

After this, we are ready to prove Theorem \ref{small angle theorem without consistency assumption theorem} above:
\begin{proof}
Let $\beta$ be as in Lemma \ref{almost commutativity lemma} (note that $\beta \leq \beta_0$). Fix $\varepsilon >0$ and let $\varepsilon_1 > 0$, $\varepsilon_2 >0$ be constants that will be determined later. Let $\gamma_1$ be the bound of the cosine of the angles of Theorem \ref{angle between several projections theorem} that correspond to $\varepsilon_1$. Similarly, let $\gamma_2$ be the bound of the cosine of the angles of Lemma \ref{almost commutativity lemma} that correspond to $\varepsilon_2$. Choose $\gamma = \min \lbrace \gamma_1, \gamma_2 \rbrace$. 

Let $\tau, \tau ' \subseteq \triangle$. Without loss of generality, it is sufficient to show that
$$\Vert P_\tau P_{\tau'} (I-P_{\tau \cap \tau'}) \Vert \leq \varepsilon.$$
Note that by the same arguments of the proof of Theorem \ref{small angle theorem}, we can assume that $\tau \cap \tau ' \in \triangle (n-1-k)$ with $0 \leq k \leq n$.

Let $\sigma_0,...,\sigma_k \in \triangle (n-1)$ be the pairwise disjoint simplices that contain $\tau \cap \tau'$. Without loss of generality we can assume that
$$\tau \subseteq \sigma_0,...,\tau \subseteq \sigma_j \text{ and } \tau' \subseteq \sigma_{j+1},...,\tau' \subseteq \sigma_k.$$
We note that
$$P_\tau = P_{\sigma_0} ... P_{\sigma_j} P_\tau ,$$
and
$$P_{\tau'} = P_{\sigma_{j+1}} ... P_{\sigma_k} P_{\tau'}.$$
Therefore
$$P_\tau P_{\tau'} (I-P_{\tau \cap \tau'}) = P_{\sigma_0} ... P_{\sigma_j} P_\tau P_{\sigma_{j+1}} ... P_{\sigma_k} P_{\tau'} (I-P_{\tau \cap \tau'}).$$
We note that 
\begin{dmath*}
\Vert P_{\sigma_0} ... P_{\sigma_j} P_\tau P_{\sigma_{j+1}} ... P_{\sigma_k} P_{\tau'} (I-P_{\tau \cap \tau'}) \Vert \leq \\
\Vert P_{\sigma_0} ... P_{\sigma_{j-1}} (P_{\sigma_{j}} P_\tau -P_\tau P_{\sigma_{j}}) P_{\sigma_{j+1}} ... P_{\sigma_k} P_{\tau'} (I-P_{\tau \cap \tau'}) \Vert + \\
\Vert P_{\sigma_0} ... P_{\sigma_{j-1}} P_\tau P_{\sigma_{j}} P_{\sigma_{j+1}} ... P_{\sigma_k} P_{\tau'} (I-P_{\tau \cap \tau'}) \Vert \leq \\
(4(n+1)+2)^{k+1} (4(n+1)+3) \varepsilon_2 + \\
\Vert P_{\sigma_0} ... P_{\sigma_{j-1}} P_\tau P_{\sigma_{j}} P_{\sigma_{j+1}} ... P_{\sigma_k} P_{\tau'} (I-P_{\tau \cap \tau'}) \Vert,
\end{dmath*}
where the last inequality is due to Lemma \ref{almost commutativity lemma}. Applying the same argument several times, we get that 
\begin{dmath*}
\Vert P_\tau P_{\tau'} (I-P_{\tau \cap \tau'}) \Vert = \Vert P_{\sigma_0} ... P_{\sigma_j} P_\tau P_{\sigma_{j+1}} ... P_{\sigma_k} P_{\tau'} (I-P_{\tau \cap \tau'}) \Vert \leq \\ 
(j+2)(4(n+1)+2)^{k+1} (4(n+1)+3) \varepsilon_2 + \Vert  P_\tau P_{\sigma_0} ... P_{\sigma_k} (I-P_{\tau \cap \tau'}) P_{\tau'}  \Vert \leq \\
(j+2)(4(n+1)+2)^{k+1} (4(n+1)+3) \varepsilon_2 +  (4(n+1)+2)^2 \cos (\angle (P_{\sigma_0},...,P_{\sigma_k})) \leq \\
(j+2)(4(n+1)+2)^{k+1} (4(n+1)+3) \varepsilon_2 +  (4(n+1)+2)^2 \varepsilon_1 \leq \\
(n+2)(4(n+1)+2)^{n+1} (4(n+1)+3) \varepsilon_2 +  (4(n+1)+2)^2 \varepsilon_1
\end{dmath*}

Therefore choosing 
$$\varepsilon_1 = \frac{\varepsilon}{2(4(n+1)+2)^2}, \varepsilon_2 = \frac{\varepsilon}{2(n+2)(4(n+1)+2)^{n+1} (4(n+1)+3)},$$ 
yields the needed inequality.

\end{proof}

\bibliographystyle{plain}
\bibliography{bibl}

\def\cprime{$'$} \def\polhk#1{\setbox0=\hbox{#1}{\ooalign{\hidewidth
  \lower1.5ex\hbox{`}\hidewidth\crcr\unhbox0}}}
\begin{thebibliography}{10}

\bibitem{BuildingsBook}
Peter Abramenko and Kenneth~S. Brown.
\newblock {\em Buildings}, volume 248 of {\em Graduate Texts in Mathematics}.
\newblock Springer, New York, 2008.
\newblock Theory and applications.

\bibitem{BallmannS}
W.~Ballmann and J.~{\'S}wi{\polhk{a}}tkowski.
\newblock On {$L^2$}-cohomology and property ({T}) for automorphism groups of
  polyhedral cell complexes.
\newblock {\em Geom. Funct. Anal.}, 7(4):615--645, 1997.

\bibitem{PropTBook}
Bachir Bekka, Pierre de~la Harpe, and Alain Valette.
\newblock {\em Kazhdan's property ({T})}, volume~11 of {\em New Mathematical
  Monographs}.
\newblock Cambridge University Press, Cambridge, 2008.

\bibitem{BenLin}
Yoav Benyamini and Joram Lindenstrauss.
\newblock {\em Geometric nonlinear functional analysis. {V}ol. 1}, volume~48 of
  {\em American Mathematical Society Colloquium Publications}.
\newblock American Mathematical Society, Providence, RI, 2000.

\bibitem{InterpolationSpaces}
J{\"o}ran Bergh and J{\"o}rgen L{\"o}fstr{\"o}m.
\newblock {\em Interpolation spaces. {A}n introduction}.
\newblock Springer-Verlag, Berlin-New York, 1976.
\newblock Grundlehren der Mathematischen Wissenschaften, No. 223.

\bibitem{BorelW}
A.~Borel and N.~Wallach.
\newblock {\em Continuous cohomology, discrete subgroups, and representations
  of reductive groups}, volume~67 of {\em Mathematical Surveys and Monographs}.
\newblock American Mathematical Society, Providence, RI, second edition, 2000.

\bibitem{LaatSalle2}
Tim de~Laat and Mikael de~la Salle.
\newblock Approximation properties for noncommutative $l^p$-spaces of high rank
  lattices and nonembeddability of expanders.
\newblock \url{http://arxiv.org/abs/1403.6415}, 2014.

\bibitem{DymaraJ}
Jan Dymara and Tadeusz Januszkiewicz.
\newblock Cohomology of buildings and their automorphism groups.
\newblock {\em Invent. Math.}, 150(3):579--627, 2002.

\bibitem{ErshovJZ}
Mikhail Ershov and Andrei Jaikin-Zapirain.
\newblock Property ({T}) for noncommutative universal lattices.
\newblock {\em Invent. Math.}, 179(2):303--347, 2010.

\bibitem{FeitHig}
Walter Feit and Graham Higman.
\newblock The nonexistence of certain generalized polygons.
\newblock {\em J. Algebra}, 1:114--131, 1964.

\bibitem{James}
Robert~C. James.
\newblock Uniformly non-square {B}anach spaces.
\newblock {\em Ann. of Math. (2)}, 80:542--550, 1964.

\bibitem{John}
Fritz John.
\newblock Extremum problems with inequalities as subsidiary conditions.
\newblock In {\em Studies and {E}ssays {P}resented to {R}. {C}ourant on his
  60th {B}irthday, {J}anuary 8, 1948}, pages 187--204. Interscience Publishers,
  Inc., New York, N. Y., 1948.

\bibitem{Kassabov}
Martin Kassabov.
\newblock Subspace arrangements and property {T}.
\newblock {\em Groups Geom. Dyn.}, 5(2):445--477, 2011.

\bibitem{Megre}
Michael~G. Megrelishvili.
\newblock Fragmentability and continuity of semigroup actions.
\newblock {\em Semigroup Forum}, 57(1):101--126, 1998.

\bibitem{MilmanWolfson}
V.~D. Mil{\cprime}~man and H.~Wolfson.
\newblock Minkowski spaces with extremal distance from the {E}uclidean space.
\newblock {\em Israel J. Math.}, 29(2-3):113--131, 1978.

\bibitem{Nowak2}
Piotr~W. Nowak.
\newblock Group actions on {B}anach spaces.
\newblock In {\em Handbook of group actions. {V}ol. {II}}, volume~32 of {\em
  Adv. Lect. Math. (ALM)}, pages 121--149. Int. Press, Somerville, MA, 2015.

\bibitem{ORobust}
Izhar Oppenheim.
\newblock Averaged projections, angles between groups and strengthening of
  property ({T}).
\newblock {\em Mathematische Annalen}, To appear., 2016.

\bibitem{HistoryBanach}
Albrecht Pietsch.
\newblock {\em History of {B}anach spaces and linear operators}.
\newblock Birkh\"auser Boston, Inc., Boston, MA, 2007.

\bibitem{Pisier}
G.~Pisier.
\newblock Some applications of the complex interpolation method to {B}anach
  lattices.
\newblock {\em J. Analyse Math.}, 35:264--281, 1979.

\bibitem{PisierXu}
Gilles Pisier and Quan~Hua Xu.
\newblock Random series in the real interpolation spaces between the spaces
  {$v_p$}.
\newblock In {\em Geometrical aspects of functional analysis (1985/86)}, volume
  1267 of {\em Lecture Notes in Math.}, pages 185--209. Springer, Berlin, 1987.

\bibitem{PisierXu2}
Gilles Pisier and Quanhua Xu.
\newblock Non-commutative {$L^p$}-spaces.
\newblock In {\em Handbook of the geometry of {B}anach spaces, {V}ol.\ 2},
  pages 1459--1517. North-Holland, Amsterdam, 2003.

\bibitem{Salle}
Mikael de~la Salle.
\newblock Towards strong {B}anach property ({T}) for {SL{$(3,\Bbb{R})$}}.
\newblock {\em Israel J. Math.}, 211(1):105--145, 2016.

\bibitem{Shalom}
Yehuda Shalom.
\newblock The algebraization of {K}azhdan's property ({T}).
\newblock In {\em International {C}ongress of {M}athematicians. {V}ol. {II}},
  pages 1283--1310. Eur. Math. Soc., Z\"urich, 2006.

\bibitem{Tomczak-Jaegermann}
Nicole Tomczak-Jaegermann.
\newblock Computing {$2$}-summing norm with few vectors.
\newblock {\em Ark. Mat.}, 17(2):273--277, 1979.

\bibitem{VanM}
Hendrik Van~Maldeghem.
\newblock {\em Generalized polygons}.
\newblock Modern Birkh\"auser Classics. Birkh\"auser/Springer Basel AG, Basel,
  1998.
\newblock [2011 reprint of the 1998 original] [MR1725957].

\bibitem{Yost}
D.~Yost.
\newblock Asplund spaces for beginners.
\newblock {\em Acta Univ. Carolin. Math. Phys.}, 34(2):159--177, 1993.
\newblock Selected papers from the 21st Winter School on Abstract Analysis
  (Pod{\v{e}}brady, 1993).

\bibitem{Zuk}
Andrzej {\.Z}uk.
\newblock La propri\'et\'e ({T}) de {K}azhdan pour les groupes agissant sur les
  poly\`edres.
\newblock {\em C. R. Acad. Sci. Paris S\'er. I Math.}, 323(5):453--458, 1996.

\end{thebibliography}

\end{document}